\documentclass{amsart}
\usepackage[margin=1in]{geometry}
\usepackage{amsthm}
\usepackage{amssymb}
\usepackage{enumitem}
\usepackage{mathtools}
\usepackage{xcolor}
\usepackage{hyperref}

\newcommand{\vol}{\mathrm{Vol}}
\newcommand{\C}{\mathcal{C}}
\newcommand{\R}{\mathbb{R}}
\newcommand{\N}{\mathbb{N}}
\newcommand{\hess}{\mathrm{Hess}}
\newcommand{\Ric}{\mathrm{Ric}}
\newcommand{\Spin}{\mathrm{Spin}}
\newcommand{\SO}{\mathrm{SO}}
\newcommand{\Cl}{\mathbb{C}\mathrm{l}}
\newcommand{\del}{\partial}
\newcommand{\hatdot}{\hat{\,\cdot\,}}
\renewcommand{\O}{\mathcal{O}}

\newtheorem{theorem}{Theorem}[section]
\newtheorem{corollary}[theorem]{Corollary}
\newtheorem{lemma}[theorem]{Lemma}
\newtheorem{definition}[theorem]{Definition}
\newtheorem{proposition}[theorem]{Proposition}
\newtheorem*{remark}{Remark}

\numberwithin{equation}{section}
 
\hypersetup{
    colorlinks,
    linkcolor={red!70!black},
    citecolor={blue!80!black},
}

\mathtoolsset{showonlyrefs}

\title[Positive mass and Dirac operators on weighted manifolds and SMMS's]{Positive mass and Dirac operators on weighted manifolds and smooth metric measure spaces}
\author[Law]{Michael B. Law}
\author[Lopez]{Isaac M. Lopez}
\author[Santiago]{Daniel Santiago}
\address{MIT, Department of Mathematics, 77 Massachusetts Avenue, Cambridge, MA 02139, USA.}
\email{mikelaw@mit.edu, imlopez@mit.edu, dsantiag@mit.edu}

\begin{document}

\begin{abstract}
    We establish a weighted positive mass theorem which unifies and generalizes results of Baldauf--Ozuch and Chu--Zhu. Our result is in fact equivalent to the usual positive mass theorem, and can be regarded as a positive mass theorem for smooth metric measure spaces. We also study Dirac operators on certain warped product manifolds associated to smooth metric measure spaces. Applications of this include, among others, an alternative proof for a special case of our positive mass theorem, eigenvalue bounds for the Dirac operator on closed spin manifolds, and a new way to understand the weighted Dirac operator using warped products.
\end{abstract}

\maketitle

\section{Introduction} \label{sec:intro}

One of the crowning achievements in contemporary geometry is the positive mass theorem in general relativity. We recall the Riemannian version of this theorem, referring to \S\ref{subsec:asymp-Euc} for the definitions of asymptotically Euclidean (AE) manifolds and ADM mass used in this paper.
\begin{theorem}[\cite{schoen1979PMT1,schoen1979PMT2,witten1981new}] \label{thm:PMT}
    Let $(M^n,g)$, $n \geq 3$ be an AE manifold of order $\tau > \frac{n-2}{2}$, and assume that $3 \leq n \leq 7$ or $M$ is spin. If $(M^n,g)$ has nonnegative scalar curvature $R \geq 0$, then it has nonnegative ADM mass $\mathfrak{m}(g) \geq 0$, with equality if and only if $(M^n,g)$ is isometric to $(\R^n,\delta_{ij})$.
\end{theorem}
The positive mass theorem was proved by Schoen and Yau in the $3 \leq n \leq 7$ case and by Witten in the spin case.\footnote{Mathematical accounts of Witten's proof can be found in \cite{lee1987yamabe}, or in \cite{parker1982Witten} for the spacetime case.}
Among the many related developments that have since arisen, the most relevant to this paper are the generalizations of the positive mass theorem to \emph{weighted manifolds}. These were recently established by Baldauf--Ozuch \cite{baldauf2022spinors} in the spin case and Chu--Zhu \cite{chu2023nonspin} in the $3 \leq n \leq 7$ case.

A weighted manifold $(M^n,g,f)$ is a Riemannian manifold $(M^n,g)$ with a weight $f \in \C^\infty(M)$ which defines the measure $e^{-f}d\vol_g$. They were first studied by Lichnerowicz \cite{lichnerowicz1970riem,lichnerowicz1971kah} and appear in many parts of mathematics, such as in Ricci flow thanks to the work of Perelman \cite{perelman2002entropy}. Perelman characterized Ricci flow as the gradient flow of the functional $\mathcal{F}$ defined on weighted manifolds by
\begin{equation}
    \mathcal{F}(M,g,f) := \int_M R_f e^{-f} d\vol_g,
\end{equation}
where $R_f = R + 2\Delta f - |\nabla f|^2$ is the \emph{weighted scalar curvature} (or \emph{$P$-scalar curvature}). Besides, weighted manifolds find applications in physics through the Brans--Dicke theory of scalar-tensor gravitation \cite{woolgar2013scalar,galloway2014cosmological} as well as theories involving Kaluza--Klein compactifications \cite{deluca2021leaps}.

Various facts about manifolds with positive (resp. nonnegative) scalar curvature are known to generalize to weighted manifolds with positive (resp. nonnegative) weighted scalar curvature. Results of this flavor include those of e.g. \cite{fan2008topology,abedin2017p,deng2021curvature}, as well as the weighted positive mass theorem described next. In \cite{baldauf2022spinors}, Baldauf and Ozuch define the \emph{weighted mass} $\mathfrak{m}_f(g)$ of an AE weighted manifold $(M^n,g,f)$ by
\begin{equation} \label{eq:weighted-mass}
    \mathfrak{m}_f(g): = \mathfrak{m}(g) + 2\lim_{\rho \rightarrow \infty}\int_{S_\rho}\langle \nabla f,\mathbf{n} \rangle e^{-f}dA,
\end{equation}
where $S_\rho$ is a coordinate sphere in the end of $M$, $\mathbf{n}$ is the Euclidean outward unit normal, and $dA$ is the Euclidean area element. The weighted positive mass theorem reads as follows. (The weighted H\"older spaces $\C^{k,\alpha}_{\beta}(M)$ are defined in Definition \ref{defn:weightedHolder}.)
\begin{theorem}[\cite{baldauf2022spinors,chu2023nonspin}] \label{thm:weightedPMT}
    Let $(M^n,g,f)$, $n \geq 3$ be an AE weighted manifold of order $\tau > \frac{n-2}{2}$, and assume $f \in \C^{2,\alpha}_{-\tau}(M)$ and $R_f \in L^1(M,g)$.
    \begin{enumerate}[label=(\alph*)]
        \item Suppose $M$ is spin and $R_f \geq 0$. Then $\mathfrak{m}_f(g) \geq 0$, with equality if and only if $(M^n,g)$ is isometric to $(\R^n,\delta_{ij})$ and $\int_{\R^n} (\Delta_f f)e^{-f} \, dx = 0$.
        \item Suppose $3 \leq n \leq 7$ and $R_f \geq 0$, with slightly more decay on $f$ and $g$ (omitting details). Then $\mathfrak{m}_f(g) \geq 0$, with equality if and only if $(M^n,g)$ is isometric to $(\R^n,\delta_{ij})$ and $f \equiv 0$.
    \end{enumerate}
\end{theorem}
Baldauf and Ozuch prove (a) by adapting Witten's proof to the weighted setting, while Chu and Zhu prove (b) by adapting Schoen and Yau's argument. We elaborate on the former for the sake of later discussion. Recall that Witten's proof hinges on finding a spinor $\phi$ on $M$ such that
\begin{enumerate}[label=(\roman*)]
    \item $D\phi = 0$, where $D$ is the Dirac operator.
    \item $\phi$ asymptotes to a constant unit norm spinor at infinity.
\end{enumerate}
A spinor satisfying (i) and (ii) is called a \emph{Witten spinor}.
These properties are used to prove \emph{Witten's formula} for the mass in terms of $\phi$ and the scalar curvature $R$:
\begin{equation} \label{eq:WittenFormula}
    \mathfrak{m}(g) = 4\int_M \left(|\nabla \phi|^2 + \frac{1}{4}R|\phi|^2\right)d\vol_g.
\end{equation}
Thus $\mathfrak{m}(g) \geq 0$ if $R \geq 0$. Baldauf and Ozuch proceed similarly to prove Theorem \ref{thm:weightedPMT}(a); they find a \emph{weighted Witten spinor} $\phi$ with properties analogous to (i) and (ii):
\begin{enumerate}[label=(\roman*')]
    \item $D_f\phi = 0$, where $D_f$ is the \emph{weighted Dirac operator}
    \begin{equation} \label{eq:weighted-Dirac}
        D_f = D - \frac{1}{2}\nabla f \cdot.
    \end{equation}
    \item $\phi$ asymptotes to a constant unit norm spinor at infinity.
\end{enumerate}
These properties lead to a formula for the weighted mass analogous to \eqref{eq:WittenFormula}:
\begin{equation} \label{eq:WeightedWittenFormula}
    \mathfrak{m}_f(g) = 4\int_M \left(|\nabla \phi|^2 + \frac{1}{4}R_f|\phi|^2\right) e^{-f} d\vol_g,
\end{equation}
from which Theorem \ref{thm:weightedPMT}(a) follows.

We now introduce the main results of this paper.

\subsection{Equivalence between weighted and unweighted positive mass theorems} \label{subsec:equivalence}

The unweighted positive mass theorem (Theorem \ref{thm:PMT}) is the $f=0$ case of the weighted positive mass theorem (Theorem \ref{thm:weightedPMT}). Our first result says that these theorems are actually \emph{equivalent}. We also sharpen the conclusions of Theorem \ref{thm:weightedPMT}:
\begin{theorem} \label{thm:main-consolidated}
    Let $(M^n,g,f)$, $n \geq 3$ be an AE weighted manifold of order $\tau > \frac{n-2}{2}$, and assume $f \in \C^{2,\alpha}_{-\tau}(M)$ and $R_f \in L^1(M,g)$. Also suppose $3 \leq n \leq 7$ or $M$ is spin.
    \begin{enumerate}[label=(\alph*)]
        \item If $R_f \geq 0$, then $\mathfrak{m}_f(g) \geq 0$, with equality if and only if $(M^n,g)$ is isometric to $(\R^n,\delta_{ij})$ and $f \equiv 0$. 
        \item The result of part (a) is equivalent to the unweighted positive mass theorem (Theorem \ref{thm:PMT}).
        \item If $R_f \geq -\frac{1}{n-1}|\nabla f|^2$, then we still have $\mathfrak{m}_f(g) \geq 0$.
    \end{enumerate}
\end{theorem}
Theorem \ref{thm:main-consolidated}(a) improves on Theorem \ref{thm:weightedPMT} due to the sharper rigidity in the spin case. Theorem \ref{thm:main-consolidated}(b) follows from the proof of part (a), which uses a suitable conformal change of metric to reduce to the unweighted positive mass theorem. Theorem \ref{thm:main-consolidated}(c) follows directly from the associated computations, and further strengthens the result of part (a) by weakening the lower bound on $R_f$ that guarantees nonnegativity of the weighted mass.

\subsection{A positive mass theorem for smooth metric measure spaces}

Generalizing beyond weighted manifolds, one arrives at \emph{smooth metric measure spaces} (SMMSs).
The following definition is from \cite{case2012smooth}, which introduces SMMSs more thoroughly and unifies the perspectives of Bakry--\'Emery, Chang--Gursky--Yang, and Perelman on the subject.
\begin{definition}
    A \emph{smooth metric measure space} (SMMS) is a 4-tuple $\mathcal{M} = (M^n,g,e^{-f} d\vol_g,m)$ where $(M^n,g)$ is a Riemannian manifold, $e^{-f}d\vol_g$ is a measure defined by a weight $f \in \C^\infty(M)$, and $m \in \R$.
\end{definition}

The analogs of Ricci and scalar curvatures on an SMMS with $m \neq 0$ are the \emph{$m$-Bakry--Émery Ricci} and \emph{scalar curvatures}, respectively
\begin{align}
    \Ric^m_f &:= \Ric + \hess_f - \frac{1}{m}df \otimes df, \label{eq:m-weighted-ricci} \\
    R^m_f &:= R_f - \frac{1}{m}|\nabla f|^2 = R + 2\Delta f - \frac{m+1}{m}|\nabla f|^2. \label{eq:m-weighted-scalar}
\end{align}
As $m \to \infty$, we have $R^m_f \to R_f$ and $\Ric^m_f \to \Ric_f$, where $\Ric_f = \Ric + \hess_f$ is commonly called the \emph{Bakry--\'Emery Ricci curvature}. Because $R_f$ and $\Ric_f$ typically appear in the context of weighted manifolds $(M,g,f)$, e.g. in Ricci flow, weighted manifolds can be viewed as SMMSs with $m=\infty$. Other values of $m$ bear significance in geometry and physics (see e.g. \cite{case2012smooth,woolgar2016cosmological} and \cite[pp.1081--82]{wu2016weitzenbock}). Of special importance to us are SMMSs with $m \in \N$, in which case $\Ric^m_f$ and $R^m_f$ arise from certain warped products over $(M,g)$. Specifically, if $(F^m,h)$ is an $m$-dimensional scalar-flat manifold, then the warped product $(M^n \times F^m, \bar{g} = g \oplus e^{-\frac{2f}{m}}h)$ has scalar curvature $R_f^m$,
and its Ricci tensor satisfies $\Ric_{\overline{g}}(X,Y) = \Ric_f^m(X,Y)$ for all vector fields $X, Y \in TM$ \cite[Proposition 9.106]{besse2007einstein}.

\begin{definition} \label{def:SMMS-AE-and-mass}
    An SMMS $\mathcal{M} = (M^n,g,e^{-f}d\vol_g,m)$ is called \emph{asymptotically Euclidean} (AE) if $(M^n,g)$ is AE. If $\mathcal{M}$ is an AE SMMS, then its \emph{mass} is defined to simply be the weighted mass of $(M^n,g,f)$:
    \begin{equation} \label{eq:equality-of-masses}
        \mathfrak{m}(\mathcal{M}) := \mathfrak{m}_f(g).
    \end{equation}    
\end{definition}

The mass is curiously independent of $m$, but is motivated as follows. If $m \in \N$, then there is a close connection between the SMMS $\mathcal{M}$ and the warped product $(M^n \times F^m, \bar{g} = g \oplus e^{-\frac{2f}{m}}h)$ as discussed above. Normalizing so that $(F^m,h)$ has unit volume, it is natural to define the mass of $\mathcal{M}$ similarly to the mass of an AE manifold:
\begin{equation} \label{eq:mass-SMMS-1}
    \mathfrak{m}(\mathcal{M}) := \lim_{\rho\to\infty} \int_{S^M_\rho \times F} (\del_i \bar{g}_{ij} - \del_j \bar{g}_{aa}) \mathbf{n}_j \, dA \, d\vol_h.
\end{equation}
(See Definition \ref{def:SMMS-mass} for a precise definition.) We will show that this coincides with the weighted mass $\mathfrak{m}_f(g)$. This leads us to define the mass of an AE SMMS, with $m$ not necessarily in $\N$, as its weighted mass.

As $R^m_f$ \eqref{eq:m-weighted-scalar} is the analog of scalar curvature on an SMMS, a positive mass theorem for SMMSs should assert that an AE SMMS with $R^m_f \geq 0$ has $\mathfrak{m}(\mathcal{M}) \geq 0$. Using Theorem \ref{thm:main-consolidated}, we will see that such a theorem does indeed hold for $m$ outside the interval $(1-n,0]$. This is our second main result.
\begin{theorem} \label{thm:main-SMMS-PMT}
    Let $\mathcal{M} = (M^n,g,e^{-f}d\vol_g,m)$ be an AE SMMS of order $\tau > \frac{n-2}{2}$ with $m \in \R \setminus (1-n,0]$, and suppose $3 \leq n \leq 7$ or $M$ is spin. Also assume $f \in \C^{2,\alpha}_{-\tau}(M)$ and $R_f \in L^1(M,g)$. If $R^m_f \geq 0$, then $\mathfrak{m}(\mathcal{M}) \geq 0$, with equality if and only if $(M^n,g)$ is isometric to $(\R^n,\delta_{ij})$ and $f \equiv 0$.
\end{theorem}
\begin{remark}
    In \cite{dai2004positive}, Dai proved a positive mass theorem for spin manifolds asymptotic to $(\R^n \setminus \overline{B_R(0)}) \times X$, where $X$ is a Calabi--Yau (hence scalar-flat) manifold. Dai defines a mass for such manifolds, which coincides with \eqref{eq:mass-SMMS-1} up to a factor if the manifold is \emph{globally} the product of an AE manifold and $X$. Therefore, Theorem \ref{thm:main-SMMS-PMT} implies Dai's positive mass theorem in this special case. The advantange of this approach is that it avoids the Mazzeo--Melrose fibered boundary calculus that Dai used to prove the general case of his result.
\end{remark}

\subsection{Warped product Dirac operators and applications} \label{subsec:intro-warped}

Let $\mathcal{M} = (M^n,g,e^{-f}d\vol_g,m)$ be a SMMS with $m \in \N$, and suppose $M$ is spin. Let $(F^m,h)$ be a spin manifold and form the warped product
\begin{equation} \label{eq:warped}
    (M^n \times F^m,\overline{g} = g \oplus e^{-\frac{2f}{m}}h).
\end{equation}
Using \cite[\S 3.2]{roos2020dirac}, we will identify the spinor bundle $\bar{\Sigma}(M \times F)$ of the warped product in terms of the spinor bundles $\Sigma M$ and $\Sigma F$, and describe the Dirac operator on $\bar{\Sigma}(M \times F)$ in terms of operators on $\Sigma M$ and $\Sigma F$. For illustration purposes, assume that $M$ is even-dimensional. We will see that $\bar{\Sigma}(M \times F) \cong \Sigma M \otimes \Sigma F$ as vector bundles, and that the Dirac operator on $\bar{\Sigma}(M \times F)$ acts in the following way. The next theorem is a special case of Theorem \ref{thm:warped-dirac}, our third main result.
\begin{theorem} \label{thm:dirac-op-intro}
    Let $M$ be even-dimensional, $\phi \in \Gamma(\Sigma M)$ be a spinor on $M$ and $\nu \in \Gamma(\Sigma F)$ be a parallel spinor, i.e. $\nabla\nu = 0$. Then the Dirac operator $\bar{D}$ on $\bar{\Sigma}(M \times F) \cong \Sigma M \otimes \Sigma F$ satisfies
    \begin{equation}
        \bar{D}(\phi \otimes \nu) = (D_f\phi) \otimes \nu,
    \end{equation}
    where $D_f$ is the weighted Dirac operator on $\Sigma M$, defined in \eqref{eq:weighted-Dirac}.
\end{theorem}
Thus, the weighted Dirac operator $D_f$ arises naturally from the SMMS $\mathcal{M}$ associated to $M$. This stands in contrast to the motivations for $D_f$ provided by Perelman \cite{perelman2002entropy} and Baldauf--Ozuch \cite{baldauf2022spinors} (where $D_f$ is an \emph{ad hoc} operator allowing identities like the Lichnerowicz formula to be generalized to weighted manifolds), and by Branding and Habib \cite{branding2022eigenvalue} (who exhibit $D_f$ as the Euler--Lagrange operator of a spinorial energy involving the weighted measure $e^{-f}d\vol_g$).

Intriguingly, the fiber dimension $m = \dim F$ does not appear in Theorem \ref{thm:dirac-op-intro}. This may have interesting geometric consequences beyond this paper, and possibly consequences in physics too. If $F$ has special holonomy except quaternionic K\"ahler, then it has nontrivial parallel spinors \cite{wang1989parallel}, so Theorem \ref{thm:dirac-op-intro} can be meaningfully applied to the warped product $(M \times F, \bar{g})$. Such products, and more generally geometries with dimensions `hidden' in special holonomy fibers, are frequently encountered in string theories, so the fiber-dimension independence in Theorem \ref{thm:dirac-op-intro} might prove significant in those contexts. We remark that Dirac operators on product manifolds have been investigated in the setting of collapsing fibrations \cite{ammann1998dirac,lott2002collapsing,roos2020dirac} which is the essence of Kaluza--Klein dimensional reduction. However, to our knowledge, existing work does not examine the dependence of the results (or lack thereof) on the fiber dimension.

This paper gives several applications of Theorem \ref{thm:dirac-op-intro} and its surrounding computations. Firstly, we will use Theorem \ref{thm:dirac-op-intro} to turn facts from unweighted spin geometry on the warped product into facts in weighted spin geometry on $M$. This allows us to systematically reprove some known identities in weighted spin geometry, one being the weighted Witten formula \eqref{eq:WeightedWittenFormula} which in turn gives a second proof of a special case of Theorem \ref{thm:main-SMMS-PMT}.
Other applications pertain to the spectrum of the Dirac operator on closed manifolds, such as using Theorem \ref{thm:dirac-op-intro} to give eigenvalue bounds in terms of $R_f^m$.
Moreover, we will leverage a relationship between the weighted Dirac operator $D_f$ and the conformal metric used to prove Theorem \ref{thm:main-consolidated} to generalize the classical fact that a closed spin manifold with positive scalar curvature admits no nontrivial harmonic spinors:
\begin{corollary} \label{cor:vanishing-thm}
    Let $M$ be a closed spin manifold. If for some $m \in \R \setminus [1-n,0]$ we have $R^m_f \geq 0$ and $R^m_f > 0$ at some point, then $M$ admits no nontrivial harmonic spinors.
\end{corollary} 
For each $m \in \R \setminus \{0\}$, there is a unique $\mu_m \in \R$ such that $R^m_f = \mu_m$ has a solution $f \in \C^\infty(M)$. While we are unable to combine this with Corollary \ref{cor:vanishing-thm} to get new obstructions to harmonic spinors, we find that the $\mu_m$ yield a family of inequalities $\lambda_1(D)^2 \geq \frac{n}{4(n-1)}\mu_m$ for the lowest eigenvalue of the Dirac operator, interpolating between the (stronger) Friedrich and Hijazi inequalities \cite{friedrich1980eigen,hijazi1986conformal} as $m$ varies.

\subsection*{Organization}

In Section \ref{sec:PMTs}, after stating definitions and conventions, we prove our positive mass theorems, Theorem \ref{thm:main-consolidated} and Theorem \ref{thm:main-SMMS-PMT}. We also motivate our definition of mass for SMMSs. In Section \ref{sec:spin}, we study the spin geometry of warped products, eventually relating the connection and Dirac operator on $(M \times F,\bar{g})$ to the corresponding objects on $M$ and $F$ in Theorem \ref{thm:warped-dirac}. In Section \ref{sec:applications-Dirac} we discuss applications of Theorem \ref{thm:warped-dirac} (or its special case Theorem \ref{thm:dirac-op-intro}) and other computations, as outlined above.

\subsection*{Acknowledgments} The authors thank Tristan Ozuch for suggesting the project, and for continuously giving invaluable insights and advice. M.L. was supported in part by a Croucher Scholarship. I.M.L. and D.S. were supported in part by the MIT Department of Mathematics through its Summer Program in Undergraduate Research (SPUR).

\section{Positive mass theorems} \label{sec:PMTs}

We begin in \S\ref{subsec:asymp-Euc} by stating our conventions for AE manifolds and ADM mass. We then prove our positive mass theorems, Theorem \ref{thm:main-consolidated} and Theorem \ref{thm:main-SMMS-PMT}, in \S\ref{subsec:weightedPMT} and \S\ref{subsec:PMT-for-SMMS} respectively.
In \S\ref{subsec:mass-of-SMMS}, we motivate the weighted mass $\mathfrak{m}_f(g)$ as a reasonable notion of mass for SMMSs.

\subsection{Asymptotically Euclidean manifolds and mass} \label{subsec:asymp-Euc}

We adhere to the following conventions in this paper. A detailed exposition to the concepts below may be found in \cite{lee2019geom}.
\begin{definition} \label{defn:AE}
    A complete Riemannian manifold $(M^n,g)$ of dimension $n \geq 3$ is said to be \emph{asymptotically Euclidean} (AE) of order $\tau > \frac{n-2}{2}$ if
    \begin{enumerate}[label=(\alph*)]
        \item There is a decomposition $M = M_{\mathrm{cpct}} \cup M_\infty$, where $M_{\mathrm{cpct}}$ is compact and $M_\infty$ is diffeomorphic to the complement of a closed ball in $\R^n$.
        \item In the induced asymptotic coordinates \emph{``$x$''} for $M_\infty$, we have for $\rho=|x|$ the falloff conditions
        \begin{equation} \label{g-falloff}
            g_{ij} = \delta_{ij} + \O(\rho^{-\tau}), \quad \partial_k g_{ij} = \O(\rho^{-\tau-1}), \quad \partial_k \partial_l g_{ij} = \O(\rho^{-\tau-2}).
        \end{equation}
        \item The scalar curvature $R = R_g$ belongs to $L^1(M,g)$.
    \end{enumerate}
\end{definition}
While this definition excludes the possibility of having multiple ends, the results of this paper extend to that case upon applying now-standard modifications.

\begin{definition}[\cite{arnowitt1960canonical,arnowitt1960energy,arnowitt1961coordinate}]
    The \emph{(ADM) mass} of an AE manifold $(M,g)$ is
    \begin{equation} \label{eq:mass}
        \mathfrak{m}(g):= \lim_{\rho \rightarrow \infty}\int_{S_\rho} \sum_{i,j=1}^n (\del_i g_{ij} - \del_j g_{ii})\mathbf{n}_j \, dA,
    \end{equation}
    where $S_\rho$ is the coordinate sphere of radius $\rho$ in asymptotic coordinates $M_\infty \cong \R^n \setminus \overline{B_R(0)}$, $\mathbf{n}_j$ is the $j$-th component of the outward Euclidean unit normal to $S_\rho$, and $dA$ is the Euclidean area element.
\end{definition}
According to \cite{bartnik1986mass} (see also \cite[\S 9]{lee1987yamabe}), the integral \eqref{eq:mass} is finite on an AE manifold and does not depend on the choice of asymptotic coordinates; thus the mass is an invariant of $g$. The mass is defined with different normalizing constants in various references; we have chosen to follow the convention of \cite{baldauf2022spinors}.

\begin{definition} \label{defn:weightedHolder}
    Let $(M^n,g)$ be an AE manifold with asymptotic coordinates \emph{``$x$''} on $M_\infty$, as in Definition \ref{defn:AE}. For $0 < \alpha < 1$, $k \in \N_0$ and $\beta \in \R$, the \emph{weighted H\"older space} $\C^{k,\alpha}_{\beta}$ is the space of $\C^k$ functions $u: M \to \R$ for which the norm
    \begin{equation}
        \lVert u \rVert_{\C^{k,\alpha}_\beta(M)} := \sum_{0 \leq i \leq k} \left( \sup_{x \in M_\infty} \frac{|\nabla^i u(x)|}{|x|^{\beta-i}} \right) +
        \sup_{x \in M_\infty} \frac{[\nabla^k u]_{\mathcal{C}^\alpha(B_{|x|/2}(x))}}{|x|^{\beta-(k+\alpha)}}
    \end{equation}
    is finite, where $B_{|x|/2}(x)$ is the metric ball of radius $\frac{|x|}{2}$ centered at $x$ and
    \begin{align}
        [\nabla^k u]_{\mathcal{C}^\alpha(B_{|x|/2}(x))} &:= \sup_{y,z \in B_{\frac{|x|}{2}}(x)} \frac{|\nabla^k u(y) - \nabla^k u(z)|}{|y-z|^\alpha}.
    \end{align}
    If $E$ is a smooth vector bundle over $M$ equipped with a bundle metric and connection, then the spaces of sections $\C^{k,\alpha}_{\beta}(E)$ are defined analogously.
\end{definition}

Note that if $u \in \C^{k,\alpha}_\beta(M)$, then $u = \O(|x|^\beta)$ as $|x| \to \infty$.

\subsection{A weighted positive mass theorem} \label{subsec:weightedPMT}

For a weighted manifold $(M^n,g,f)$, define the metric
\begin{equation} \label{eq:conformal-metric}
    \tilde{g} = e^{-\frac{2f}{n-1}}g.
\end{equation}
To prove Theorem \ref{thm:main-consolidated}, we need two lemmas which will enable a reduction to the unweighted positive mass theorem.
Henceforth, denote by $R$ and $\tilde{R}$ the scalar curvatures of $g$ and $\tilde{g}$ respectively, and $R_f = R +2\Delta f - |\nabla f|^2$ the weighted scalar curvature of $(M,g,f)$, where covariant derivatives are taken with respect to $g$.
\begin{lemma} \label{lem:R-of-conformal}
    Let $(M^n,g,f)$ be an AE weighted manifold of order $\tau > \frac{n-2}{2}$, such that $f \in \C^{2,\alpha}_{-\tau}(M)$ and $R_f \in L^1(M,g)$. Then $(M^n,\tilde{g})$ is an AE manifold of order $\tau$ and
    \begin{equation} \label{eq:R-of-conformal}
        \tilde{R} = e^{\frac{2}{n-1}f}\left( R_f + \frac{1}{n-1}|\nabla f|^2 \right).
    \end{equation}
\end{lemma}
\begin{proof}
    Since $f \in \C^{2,\alpha}_{-\tau}(M)$, in asymptotic coordinates for the end $M_\infty$ of $M$ we have 
    \begin{equation} \label{eq:f-decay}
        f = \O(\rho^{-\tau}), \quad \del_k f = \O(\rho^{-\tau-1}), \quad \del_k \del_l f = \O(\rho^{-\tau-2}).
    \end{equation}
    In particular,
    \begin{equation} \label{eq:e^f-decay}
        e^{-\frac{2f}{n-1}} = 1 + \O(\rho^{-\tau})
    \end{equation}
    and so
    \begin{equation}
        \tilde{g}_{ij} = e^{-\frac{2f}{n-1}}g_{ij} = (1 + \O(\rho^{-\tau}))(\delta_{ij}+\O(\rho^{-\tau})) = \delta_{ij} + \O(\rho^{-\tau}).
    \end{equation}
    Using \eqref{eq:f-decay}, \eqref{eq:e^f-decay}, and the asymptotic Euclideanness of $g$, it also follows that
    \begin{align}
        \del_k \tilde{g}_{ij} &= -\frac{2}{n-1}e^{-\frac{2f}{n-1}} (\del_k f) g_{ij} + e^{-\frac{2f}{n-1}} \del_k g_{ij} = \O(\rho^{-\tau-1}).
    \end{align}
    Similarly, $\del_k \del_l \tilde{g}_{ij} = \O(\rho^{-\tau-2})$. These are the required decay conditions on $\tilde{g}$.

    To prove \eqref{eq:R-of-conformal}, let $\varphi = e^{-\frac{n-2}{2(n-1)}f}$. We have $\tilde{g} = \varphi^{\frac{4}{n-2}}g$ and
    \begin{equation} \label{eq:Laplace-phi}
        \Delta_g \varphi = e^{-\frac{n-2}{2(n-1)}f}\left( -\frac{n-2}{2(n-1)} \Delta_g f + \left( \frac{n-2}{2(n-1)} \right)^2 |\nabla f|^2 \right).
    \end{equation}
    Recall that if $g$ is a Riemannian metric on an $n$-dimensional manifold and $\varphi$ is a smooth positive function, then the conformal metric $\varphi^{\frac{4}{n-2}}g$ has scalar curvature
    \begin{equation} \label{eq:R-of-conformal1}
        \varphi^{-\frac{n+2}{n-2}} \left( -\frac{4(n-1)}{n-2}\Delta_g \varphi + R \varphi \right).
    \end{equation}
    Using this together with \eqref{eq:Laplace-phi}, it follows that
    \begin{align}
        \tilde{R} &= e^{\frac{n+2}{2(n-1)}f} \left( -\frac{4(n-1)}{n-2} e^{-\frac{n-2}{2(n-1)}f}\left( -\frac{n-2}{2(n-1)} \Delta_g f + \left( \frac{n-2}{2(n-1)} \right)^2 |\nabla f|^2 \right) + R e^{-\frac{n-2}{2(n-1)}f} \right) \\
        &= e^{\frac{2f}{n-1}}\left( 2\Delta_g f - \frac{n-2}{n-1}|\nabla f|^2 + R \right) \\
        &= e^{\frac{2f}{n-1}} \left(R_f + \frac{1}{n-1}|\nabla f|^2 \right). \label{eq:Rtilde-formula}
    \end{align}
    Since $R_f \in L^1(M,g)$ by hypothesis, $|\nabla f|^2 = \O(\rho^{-2\tau-2})$ by \eqref{eq:f-decay}, and we have the decay \eqref{eq:e^f-decay}, the formula \eqref{eq:Rtilde-formula} implies that $\tilde{R} \in L^1(M,g)$. As $\tilde{g}$ and $g$ are asymptotically equivalent this implies $\tilde{R} \in L^1(M,\tilde{g})$. So $\tilde{g}$ is an asymptotically Euclidean metric.
\end{proof}

The conformal choice \eqref{eq:conformal-metric} yields a nice formula for the weighted mass $\mathfrak{m}_f(g)$ defined in \eqref{eq:weighted-mass}:
\begin{lemma} \label{lem:equality-of-masses}
    Let $(M^n,g,f)$ be an AE weighted manifold of order $\tau > \frac{n-2}{2}$. Assume $f \in \C^{2,\alpha}_{-\tau}(M)$ and $R_f \in L^1(M,g)$. Then the weighted mass of $(M,g,f)$ equals the unweighted mass of $(M,\tilde{g})$:
    \begin{equation}
        \mathfrak{m}_f(g)=\mathfrak{m}(\tilde{g}).
    \end{equation}
\end{lemma}
\begin{proof}
    By Lemma \ref{lem:R-of-conformal}, $(M,\tilde{g})$ is AE of order $\tau$ so $\mathfrak{m}(\tilde{g})$ is well-defined.
    Since $\tilde{g} = e^{-\frac{2f}{n-1}}g$, we have
    \begin{align}
        \mathfrak{m}(\tilde{g}) &= \lim_{\rho \rightarrow \infty}\int_{S_\rho} \sum_{i,j=1}^n (\del_i \tilde{g}_{ij} - \del_j \tilde{g}_{ii})\mathbf{n}_j \, dA \\
        &= \underbrace{\lim_{\rho \to \infty} \int_{S_\rho} e^{-\frac{2f}{n-1}} \sum_{i,j=1}^n (\del_i g_{ij} - \del_j g_{ii}) \mathbf{n}_j \, dA}_{=:I_1} - \underbrace{\frac{2}{n-1}\lim_{\rho \to \infty} \int_{S_\rho} e^{-\frac{2f}{n-1}}\sum_{i,j=1}^n ((\del_i f)g_{ij} - (\del_j f) g_{ii}) \mathbf{n}_j \, dA}_{=:I_2}. \label{eq:I_1-I_2}
    \end{align}
    Since $e^{-\frac{2f}{n-1}} = 1 + \O(\rho^{-\tau})$ and $\del_i g_{ij}, \del_j g_{ii} = \O(\rho^{-\tau-1})$, we have
    \begin{align} \label{eq:I_1}
        I_1 &= \lim_{\rho\to\infty} \int_{S_\rho} \sum_{i,j=1}^n (\del_i g_{ij} - \del_j g_{ii}) \mathbf{n}_j \, dA + \lim_{\rho\to\infty} \int_{S_\rho} \O(\rho^{-2\tau-1}) \, dA = \mathfrak{m}(g).
    \end{align}
    To handle $I_2$, use that $g_{ij} = \delta_{ij} + \O(\rho^{-\tau})$ to get
    \begin{align}
        I_2 &= \frac{2}{n-1} \lim_{\rho\to\infty} \int_{S_\rho} (1+\O(\rho^{-\tau})) \sum_{i,j=1}^n [(\del_i f)(\delta_{ij} + \O(\rho^{-\tau})) - (\del_j f)(1 + \O(\rho^{-\tau}))] \mathbf{n}_j \, dA  \\
        &= \frac{2}{n-1} \lim_{\rho\to\infty} \int_{S_\rho} (1-n)\sum_{j=1}^n (\del_j f) \mathbf{n}_j + \O(\rho^{-2\tau-1}) \, dA  \\
        &= -2\lim_{\rho\to \infty} \int_{S_\rho} \langle \nabla f, \mathbf{n} \rangle e^{-f} \, dA. \label{eq:I_2}
    \end{align}
    Using \eqref{eq:I_1} and \eqref{eq:I_2} in \eqref{eq:I_1-I_2}, we have
    \begin{align}
        \mathfrak{m}(\tilde{g}) &= I_1 - I_2 = \mathfrak{m}(g) + 2\lim_{\rho\to\infty} \int_{S_\rho} \langle \nabla f, \mathbf{n} \rangle e^{-f} \, dA = \mathfrak{m}_f(g).
    \end{align}
\end{proof}

We will now prove Theorem \ref{thm:main-consolidated} using the previous two lemmas.

\begin{proof}[Proof of Theorem \ref{thm:main-consolidated}]
    Let $(M^n,g,f)$ be a weighted manifold which is AE of order $\tau > \frac{n-2}{2}$, and suppose $3 \leq n \leq 7$ or $M$ is spin. Also assume $f \in \C^{2,\alpha}_{-\tau}(M)$ and $R_f \in L^1(M,g)$.

    If $R_f \geq -\frac{1}{n-1}|\nabla f|^2$, then Lemma \ref{lem:R-of-conformal} shows that the conformal metric $\tilde{g} = e^{-\frac{2f}{n-1}}g$ has scalar curvature $\tilde{R} \geq 0$. By Lemma \ref{lem:equality-of-masses} and Theorem \ref{thm:PMT},
    \begin{equation} \label{eq:PMT-proof-1}
        \mathfrak{m}_f(g) = \mathfrak{m}(\tilde{g}) \geq 0.
    \end{equation}
    This proves the first claim in part (a) of Theorem \ref{thm:main-consolidated}, as well as part (c). 
    
    Now suppose $R_f \geq 0$ and $\mathfrak{m}_f(g) = 0$. Lemmas \ref{lem:R-of-conformal} and \ref{lem:equality-of-masses} then imply $\tilde{R} \geq 0$ and $\mathfrak{m}(\tilde{g}) = 0$, so the rigidity part of Theorem \ref{thm:PMT} gives that $(M^n,\tilde{g})$ is isometric to $(\R^n,\delta_{ij})$. Thus
    \begin{equation}
        0 = \tilde{R} = e^{\frac{2}{n-1}f}\left(R_f + \frac{1}{n-1}|\nabla f|^2 \right) \geq \frac{1}{n-1}e^{\frac{2}{n-1}f}|\nabla f|^2 \geq 0.
    \end{equation}
    It follows that $f$ is constant, but since $f \in \C^{2,\alpha}_{-\tau}(M)$, we have $f \equiv 0$. So $\tilde{g} = g$, and $(M^n,g)$ too is isometric to $(\R^n,\delta_{ij})$. This proves the rigidity part of Theorem \ref{thm:main-consolidated}(a).

    Theorem \ref{thm:main-consolidated}(a) implies Theorem \ref{thm:PMT} by taking the weight to be $f \equiv 0$. Conversely, we have just used Theorem \ref{thm:PMT} to prove Theorem \ref{thm:main-consolidated}(a). Hence Theorem \ref{thm:main-consolidated}(b) follows.
\end{proof}

\subsection{The mass of a smooth metric measure space} \label{subsec:mass-of-SMMS}

In Definition \ref{def:SMMS-AE-and-mass}, the mass of an AE SMMS $\mathcal{M} = (M^n,g,e^{-f}d\vol_g,m)$, $m \in \R$, was defined as the weighted mass of $(M^n,g,f)$. We will now motivate this by considering the $m \in \N$ case.

Let $\mathcal{M} = (M^n,g,e^{-f}d\vol_g,m)$ be a SMMS with $m \in \N$, let $(F^m,h)$ be an $m$-dimensional scalar-flat manifold, and form the warped product $$(M^n \times F^m,\overline{g} = g \oplus e^{-\frac{2f}{m}}h).$$ Then the warped product has scalar curvature $R^m_f$, and its Ricci tensor restricted to tangent vectors on $M$ is $\Ric^m_f$. These are the curvatures associated to $\mathcal{M}$, defined in \eqref{eq:m-weighted-ricci} and \eqref{eq:m-weighted-scalar}. We may therefore view the warped product as a natural `extrinsic' space associated to $\mathcal{M}$. For this reason, if we are to define a mass for AE SMMSs which makes use of the parameter $m$, it is natural to define it as the mass of the warped product.\footnote{Although the warped product is not AE, its mass can be defined using an integral similar to \eqref{eq:mass} as done in \eqref{eq:mass-SMMS}.} We normalize by taking $(F^m,h)$ to have unit volume.
\begin{definition} \label{def:SMMS-mass}
    Let $\mathcal{M} = (M^n,g,e^{-f}d\vol_g,m)$ be an AE SMMS with $m \in \N$. Let $(F^m,h)$ be an $m$-dimensional scalar-flat manifold of unit volume, and form the warped product $(M^n \times F^m,\overline{g} = g \oplus e^{-\frac{2f}{m}}h)$. Work in asymptotic coordinates for $(M^n,g)$ and an orthonormal frame for $(F^m,h)$. The \emph{mass} of $\mathcal{M}$ is defined as
    \begin{equation} \label{eq:mass-SMMS}
        \mathfrak{m}(\mathcal{M}) := \lim_{\rho\to\infty} \int_{S^M_\rho \times F} (\del_i \bar{g}_{ij} - \del_j \bar{g}_{aa}) \mathbf{n}_j \, dA \, d\vol_h,
    \end{equation}
    where $S^M_\rho$ is the coordinate sphere of radius $\rho$ in asymptotic coordinates for $M$, $\mathbf{n}_j$ is the $j$-th component of the outward Euclidean unit normal to $S^M_\rho$, and $dA$ is the Euclidean area element. Here the indices $i,j$ are summed over the coordinates for $M$ and the index $a$ is summed over the combined frame for $M \times F$.
\end{definition}

The next proposition reveals that \eqref{eq:mass-SMMS} is nothing but the weighted mass of $(M^n,g,f)$ (hence independent of both $m$ and $F$). This is why we have used the weighted mass as the definition of mass for SMMSs even when $m \notin \N$.

\begin{proposition} \label{prop:SMMS-mass-is-mfg}
    Let $\mathcal{M} = (M^n,g,e^{-f}d\vol_g,m)$ be an AE SMMS with $m \in \N$. Then
    \begin{equation}
        \mathfrak{m}(\mathcal{M}) = \mathfrak{m}_f(g),
    \end{equation}
    where $\mathfrak{m}_f(g)$ is the weighted mass of $(M^n,g,f)$ defined in \eqref{eq:weighted-mass}.
\end{proposition}
\begin{proof}
    Fix asymptotic coordinates for $M$, and take an orthonormal frame for $F$ so that $h$ is the $m \times m$ identity matrix. Putting these together yields a frame for the warped product $(M^n \times F^m,\overline{g} = g \oplus e^{-\frac{2f}{m}}h)$. According to Definition \ref{def:SMMS-mass}, the mass of $\mathcal{M}$ is
    \begin{align}
        \mathfrak{m}(\mathcal{M}) &=
        \lim_{\rho\to\infty} \int_{S^M_\rho \times F} (\del_i \bar{g}_{ij} - \del_j \bar{g}_{ii} - \del_j \bar{g}_{\beta\beta}) \mathbf{n}_j \, dA \, d\vol_h \label{eq:mass-computation-1}
    \end{align}
    where $i,j$ run over the coordinates for $M$ and $\beta$ runs over the coordinates for $F$. Now we have
    \begin{equation}
        \bar{g}_{ij} = g_{ij}, \quad \bar{g}_{\beta\beta} = e^{-\frac{2f}{m}} h_{\beta\beta} = e^{-\frac{2f}{m}}, \quad \del_j \bar{g}_{\beta\beta} = -\frac{2}{m}e^{-\frac{2f}{m}}\del_j f.
    \end{equation}
    Substituting these into \eqref{eq:mass-computation-1} and using the fact that $(F^m,h)$ has unit volume, we have
    \begin{align}
        \mathfrak{m}(\mathcal{M}) &= \lim_{\rho\to\infty} \int_{S^M_\rho \times F} (\del_i g_{ij} - \del_j g_{ii}) \mathbf{n}_j \, dA \, d\vol_h + \lim_{\rho\to\infty} \int_{S^M_\rho \times F} 2e^{-\frac{2f}{m}}(\del_j f) \mathbf{n}_j \, dA \, d\vol_h  \\
        &= \left(\mathfrak{m}(g) + 2 \lim_{\rho\to\infty} \int_{S^M_\rho} \langle \nabla f, \mathbf{n} \rangle e^{-f} \, dA\right) \vol(F^m,h) \\
        &= \mathfrak{m}_f(g).
    \end{align}
\end{proof}

\subsection{A positive mass theorem for smooth metric measure spaces} \label{subsec:PMT-for-SMMS}

\begin{proof}[Proof of Theorem \ref{thm:main-SMMS-PMT}]
    Let $\mathcal{M} =(M^n,g,e^{-f}d\vol_g,m)$ be an AE SMMS of order $\tau > \frac{n-2}{2}$ with $m \in \R \setminus(1-n,0]$, and suppose $3 \leq n \leq 7$ or $M$ is spin. Also assume that $$f \in \C^{2,\alpha}_{-\tau}(M), \quad R_f \in L^1(M,g), \quad R^m_f \geq 0.$$
    Since $m \notin (1-n,0]$, we have
    \begin{equation} \label{eq:Rmf-inequalit}
        0 \leq R^m_f = R_f - \frac{1}{m}|\nabla f|^2 \leq R_f + \frac{1}{n-1}|\nabla f|^2.
    \end{equation}
    It follows by Theorem \ref{thm:main-consolidated}(c) that $\mathfrak{m}(\mathcal{M}) = \mathfrak{m}_f(g) \geq 0$. If equality holds, then using the conformal metric $\tilde{g} = e^{-\frac{2f}{n-1}}g$ and Lemma \ref{lem:equality-of-masses}, we have
    \begin{equation} \label{eq:masses}
        \mathfrak{m}(\mathcal{M}) = \mathfrak{m}_f(g) = \mathfrak{m}(\tilde{g}) = 0.
    \end{equation}
    By Lemma \ref{lem:R-of-conformal}, the scalar curvature of $\tilde{g}$ is
    \begin{align} \label{eq:Rtilde-SMMS}
        \tilde{R} &= e^{\frac{2}{n-1}f}\left( R_f + \frac{1}{n-1} |\nabla f|^2 \right) = e^{\frac{2}{n-1}f}\left(R^m_f + \left(\frac{1}{m} + \frac{1}{n-1}\right)|\nabla f|^2 \right).
    \end{align}
    Combining the first equality with \eqref{eq:Rmf-inequalit} yields $\tilde{R} \geq 0$, so by \eqref{eq:masses} and the rigidity in the unweighted positive mass theorem, we have $(M,\tilde{g}) \cong (\R^n,\delta_{ij})$. Now \eqref{eq:Rtilde-SMMS} gives
    \begin{equation}
        0 = \tilde{R} = e^{\frac{2}{n-1}f}\left(R^m_f + \left(\frac{1}{m} + \frac{1}{n-1}\right)|\nabla f|^2 \right) \geq e^{\frac{2}{n-1}f}\left(\frac{1}{m}+\frac{1}{n-1}\right)|\nabla f|^2 \geq 0,
    \end{equation}
    which implies $\nabla f = 0$. But $f \in \C^{2,\alpha}_{-\tau}(M)$, so $f \equiv 0$. Thus $\tilde{g} = g$, so $(M^n,g) \cong (\R^n,\delta_{ij})$.
\end{proof}

\section{Spin geometry on warped products} \label{sec:spin}

We now turn to spinorial aspects of SMMSs. We will restrict to SMMSs $\mathcal{M} = (M^n,g,e^{-f}d\vol_g,m)$ with $m \in \N$, in which case $\mathcal{M}$ is closely associated to a warped product manifold
\begin{equation}
    (M^n \times F^m, \bar{g} = g \oplus e^{-\frac{2f}{m}}h)
\end{equation}
where $(F^m,h)$ is scalar-flat (see \S\ref{sec:intro} or \S\ref{subsec:mass-of-SMMS}). In this section, $M$ and $F$ are assumed to be spin manifolds.

We study the spin geometry of $\mathcal{M}$ via the spin geometry of the warped product $(M \times F, \bar{g})$. Doing this entails relating spinors on the warped product to spinors on the factors $(M,g)$ and $(F,h)$. Facilitated by \cite[\S 3.2]{roos2020dirac}, we carry this out in \S\ref{subsec:warped-spin-bundle} by identifying the appropriate spinor bundles, and in \S\ref{subsec:warped-Dirac} by relating their Dirac operators (Theorem \ref{thm:warped-dirac}). We begin by reviewing the necessary spin geometry; for comprehensive treatments see \cite{bourguignon2015spinorial,friedrich2000dirac,lawson2016spin}.

\subsection{Generalities on spin geometry} \label{subsec:generalities-spin}

Let $E \to X$ be a oriented vector bundle of rank $k$ equipped with a metric $h$, and let $P_{\SO}E \to X$ be the bundle of positive orthonormal bases of $(E,h)$. Recall that a \emph{spin structure} on $E$ is a principal $\Spin_k$-bundle $P_{\Spin}E \to X$ with an equivariant double cover $P_{\Spin}E \to P_{\SO}E$ with respect to the right group actions and the universal covering $\Spin_k \to \SO_k$. If $P_{\Spin}E$ is a spin structure for $(E,h)$, we can form the spinor bundle $\Sigma E = P_{\Spin}E \times_\rho \mathbb{C}^{2^{\lfloor k/2 \rfloor}}$ over $X$, where $\rho$ is the restriction to $\Spin_k$ of an irreducible representation of the complexified Clifford algebra $\Cl(\R^k)$. Then $\Sigma E$ has complex rank $2^{\lfloor \frac{k}{2}\rfloor}$, and is a bundle of modules over the bundle of Clifford algebras $\Cl(E,h)$.

Any $\Cl(E,h)$-module $S$, such as $\Sigma E$, has a Hermitian metric such that the action of unit vectors is unitary. This Hermitian metric on $S$ is obtained by an averaging procedure which we now spell out. Take an arbitrary Hermitian metric $(\cdot,\cdot)$, and for each $x \in X$ let $\Gamma_x$ be the finite subgroup of $\Cl(E_x,h_x)$ generated by an orthonormal basis for $E_x$. Now define the Hermitian metric $\langle\cdot,\cdot\rangle$ on $S$ by setting for all $\psi_1,\psi_2 \in S_x$
\begin{equation} \label{eq:averaging}
    \langle\psi_1,\psi_2\rangle = \frac{1}{|\Gamma_x|} \sum_{\tau \in \Gamma_x} (\tau \cdot \psi_1, \tau \cdot \psi_2).
\end{equation}
This averaged metric is the Hermitian metric on $S$, and is unique up to positive scaling.
Note that if we repeat the above using $\langle\cdot,\cdot\rangle$ as the starting metric, then the averaging procedure recovers $\langle\cdot,\cdot\rangle$.

A Riemannian manifold is \emph{spin} if it admits a spin structure, meaning a spin structure on its tangent bundle. Given a spin structure $P_{\Spin}(TX)$ on a spin manifold $(X^n,g)$, the spinor bundle $\Sigma X := \Sigma TX$ is a rank $2^{\lfloor \frac{n}{2} \rfloor}$ bundle of complex modules over $\Cl(TX,g)$. Moreover, $\Sigma X$ gets a Hermitian metric using the averaging procedure described above (so that Clifford multiplication by unit tangent vectors is unitary), as well as a connection defined by
\begin{equation} \label{eq:spin-connection}
    \nabla_Y \psi = d\psi(Y) + \frac{1}{4}\sum_{j,k=1}^n g(\nabla_Y e_j,e_k) e_j \cdot e_k \cdot \psi, \quad Y \in TM, \psi \in \Gamma(\Sigma X),
\end{equation}
where $e_1,\ldots,e_n$ is an orthonormal basis for $(TX,g)$. The \emph{Dirac operator} $D: \Gamma(\Sigma X) \to \Gamma(\Sigma X)$ is a symmetric first-order elliptic operator defined by
\begin{equation}
    D\psi = \sum_{i=1}^n e_i \cdot \nabla_{e_i} \psi.
\end{equation}
If $R$ denotes the scalar curvature of $(X,g)$, then $D$ satisfies the \emph{Lichnerowicz formula}
\begin{equation} \label{eq:Lich}
    D^2 \psi = -\Delta\psi + \frac{R}{4}\psi,
\end{equation}
where $\Delta = -\nabla^*\nabla = \nabla_{e_i}\nabla_{e_i} - \nabla_{\nabla_{e_i}e_i}$ is the Laplacian on the spinor bundle.

Define a section $\omega_X$ of $\Cl(TX,g)$ as follows: if $e_1,\ldots,e_n$ is a positive orthonormal basis for $(T_x X,g)$, then
\begin{equation}
    \omega_X(x) = i^{\lfloor \frac{n+1}{2} \rfloor} e_1 \cdots e_n.
\end{equation}
If $n = \dim X$ is even, then $\omega_X^2 = 1$ so $\Sigma X$ splits as an orthogonal sum of the $\pm 1$-eigenbundles:
\begin{equation}
    \Sigma X = \hat{\Sigma}^+ X \oplus \hat{\Sigma}^- X, \quad \hat{\Sigma}^\pm X = \{ \psi \in \Sigma X \mid \omega_X \cdot \psi = \pm \psi \}.
\end{equation}
We write $\psi = \psi^+ + \psi^-$ for the corresponding decomposition of $\psi \in \Sigma X$. The conjugate spinor is defined by $\bar{\psi} = \psi^+ - \psi^-$. We have $v \cdot \omega_X = -\omega_X \cdot v$ for all $v \in TX$, so Clifford multiplication by $v$ permutes $\hat{\Sigma}^\pm X$.

\subsection{The spinor bundle of a warped product} \label{subsec:warped-spin-bundle}

Let $(M^n,g)$ and $(F^m,h)$ be spin manifolds, fix a smooth function $f \in \C^\infty(M)$, and consider the warped product
\begin{equation} \label{eq:warped-1}
    (M \times F, \bar{g} = g \oplus e^{-\frac{2f}{m}}h)
\end{equation}
along with the projection maps
\begin{equation}
    \pi_1: M \times F \to M, \quad \pi_2: M \times F \to F.
\end{equation}
Fixing spin structures on $M$ and $F$ (that is, on $(TM,g)$ and $(TF,h)$ respectively), we describe the induced spin structure on $(T(M \times F), \bar{g})$. Pulling back the chosen spin structures via $\pi_1$ and $\pi_2$ gives spin structures on the bundles $(\pi_1^*TM,\pi_1^*g)$ and $(\pi_2^*TF,\pi_2^*h)$. The latter determines a spin structure for the bundle $(\pi_2^*TF, e^{-\frac{2f}{m}} \pi_2^*h)$ which is hereafter called $V$ \cite[\S 2, Remark 1.9]{lawson2016spin}. Since
\begin{equation}
    (T(M \times F),\bar{g}) \cong (\pi_1^*TM,\pi_1^*g) \oplus V,
\end{equation}
and the two summands on the right are now endowed with spin structures, their direct sum also gets a spin structure \cite[\S 2, Proposition 1.15]{lawson2016spin}. This is the induced spin structure on $(T(M \times F), \bar{g})$.

The spin structures on $(TM,g)$, $(TF,h)$, $V = (\pi_2^*TF, e^{-\frac{2f}{m}} \pi_2^*h)$ and $(T(M \times F),\bar{g})$ induce the spinor bundles $\Sigma M$, $\Sigma F$, $\Sigma V$ and $\bar{\Sigma}(M \times F)$ respectively, as listed in Table \ref{tab:spinor-bundles}.
\begin{table}
    \begin{center}
    \begin{tabular}{ |c|c|c| }
    \hline
    Spinor bundle & Base space & Module over... \\
    \hline
    $\Sigma M$ & $M$ & $\Cl(TM,g)$ \\
    $\Sigma F$ & $F$ & $\Cl(TF,h)$ \\
    $\Sigma V$ & $M \times F$ & $\Cl(\pi_2^*TF,e^{-\frac{2f}{m}}\pi_2^*h)$ \\
    $\bar{\Sigma}(M \times F)$ & $M \times F$ & $\Cl(T(M \times F),\bar{g})$ \\
    \hline
    \end{tabular}
    \end{center}
    \caption{The spinor bundles we will use.}
    \label{tab:spinor-bundles}
\end{table}
We will describe some relationships between these bundles and the structures on them. Firstly, since
$\Sigma F$ and $\pi_2^*\Sigma F$ are Clifford modules over $\Cl(TF,h)$ and $\Cl(\pi_2^*TF,\pi_2^*h)$ respectively, the averaging procedure described in \S\ref{subsec:generalities-spin} applies to give Hermitian metrics
$\langle \cdot,\cdot\rangle_{\Sigma F}$ and $\langle \cdot,\cdot \rangle_{\pi_2^*\Sigma F}$. It is not hard to see that $\langle \cdot,\cdot \rangle_{\pi_2^*\Sigma F} = \pi_2^*\langle \cdot,\cdot \rangle_{\Sigma F}$. The averaging procedure can also be used to define $\langle \cdot,\cdot \rangle_{\Sigma V}$ on $\Sigma V$; on the other hand, we have the following useful identification of $\Sigma V$ together with its inner product.
\begin{lemma} \label{lem:SigmaV}
    There is a bundle isometry $(\Sigma V, \langle \cdot,\cdot \rangle_{\Sigma V}) \cong (\pi_2^*\Sigma F, \langle \cdot,\cdot \rangle_{\pi_2^*\Sigma F})$. If \emph{``$\cdot$''} is the Clifford multiplication on $\pi_2^*\Sigma F$, then the Clifford multiplication \emph{``$\hatdot$''} on $\Sigma V$ is given under the identification $\Sigma V \cong \pi_2^*\Sigma F$ by
    \begin{equation} \label{eq:cliff-mult-V}
        \pi_2^*TF \otimes \pi_2^*\Sigma F \to \pi_2^*\Sigma F, \quad (v,\nu) \mapsto v \hatdot \nu = e^{-\frac{f}{m}} v \cdot \nu.
    \end{equation}
\end{lemma}
\begin{proof}
    This is because $V$ is obtained by multiplying the bundle metric on $(\pi_2^*TF,\pi_2^*h)$ by the conformal factor $e^{-\frac{2f}{m}}$ (see e.g. \cite[pp.69]{bourguignon2015spinorial}, whose discussion generalizes to spin structures on vector bundles).
\end{proof}
We will use ``$\cdot$'' to denote the Clifford multiplication on $\Sigma M$, $\Sigma F$, $\bar{\Sigma}(M \times F)$, and their pullbacks. In light of Lemma \ref{lem:SigmaV}, $\Sigma V$ will be understood to be the vector bundle $\pi_2^*\Sigma F$ equipped with the Clifford multiplication ``$\hatdot$'' given by \eqref{eq:cliff-mult-V}. The next proposition identifies $\bar{\Sigma}(M \times F)$ and its Clifford module structure.
\begin{proposition} \label{prop:spinor-id}
    We have
    \begin{equation} \label{eq:spinor-id-1}
        \bar{\Sigma}(M \times F) \cong \pi_1^*({}^\diamond \mathbf{\Sigma}M) \otimes \Sigma V
    \end{equation}
    as vector bundles, where
    \begin{equation} \label{eq:spinor-id-2}
        {}^\diamond \mathbf{\Sigma}M = 
        \begin{cases}
            \Sigma M & \text{if } n \text{ or } m \text{ is even}, \\
            \Sigma M \oplus \Sigma M & \text{if } n \text{ and } m \text{ are odd}.
        \end{cases}
    \end{equation}
    Under this identification, the structure of $\bar{\Sigma}(M \times F)$ as a module over $\Cl(T(M \times F),\bar{g})$ is given as follows. Let $\phi \in \pi_1^*{}^\diamond(\mathbf{\Sigma}M)$ and $\nu \in \Sigma V$. The Clifford multiplication of $\phi \otimes \nu \in \bar{\Sigma}(M \times F)$ by $(x,v) \in TM \oplus TF$ is
    \begin{equation} \label{eq:cliff-mult-MxF}
        (x,v) \cdot (\phi \otimes \nu) = \begin{cases}
            (x \cdot \phi) \otimes \nu + \bar{\phi} \otimes (v \hatdot \nu) & n \text{ even} \\
            (x \cdot \phi) \otimes \bar{\nu} + \phi \otimes (v \hatdot \nu) & n \text{ odd, } m \text{ even} \\
            (x \cdot \phi_1 \oplus -x \cdot \phi_2) \otimes \nu + (\phi_2 \oplus \phi_1) \otimes (v \hatdot \nu) & n,m \text{ odd}.
        \end{cases}
    \end{equation}
    For the last case we have written $\phi = \phi_1 + \phi_2 \in \Sigma M \oplus \Sigma M$.
\end{proposition}
\begin{proof}
    Since $\pi_1: (M \times F, \bar{g}) \to (M,g)$ is clearly a Riemannian submersion, the proposition follows from the discussion in \cite[\S 3.2]{roos2020dirac} (see specifically equations (10), (11) and Notation 1 there).
\end{proof}
One verifies that, as should be the case, applying Clifford multiplication by $(x,v)$ twice has the same effect as scaling by $-\bar{g}((x,v),(x,v))$.

The Hermitian metric on $\bar{\Sigma}(M \times F)$ is again obtained by the averaging procedure from \S\ref{subsec:generalities-spin}. Namely, using the identification \eqref{eq:spinor-id-1}, define an initial Hermitian metric on $\bar{\Sigma}(M \times F)$ by
\begin{equation} \label{eq:hermitian}
    (\phi \otimes \nu, \phi' \otimes \nu')_{\bar{\Sigma}(M \times F)} = \langle \phi, \phi' \rangle_{\pi_1^*({}^\diamond\mathbf{\Sigma}M)} \langle \nu, \nu' \rangle_{\Sigma V}
\end{equation}
and extending linearly.
For each $(x,y) \in M \times F$, let $\Gamma_{x,y}$ be the finite Clifford subgroup generated by an orthonormal basis for $(T_{(x,y)}(M \times F),\bar{g})$ of the form
\begin{equation}
    \{e_1,\ldots,e_n,\varepsilon_1,\ldots,\varepsilon_m\}
\end{equation}
where the $e_i$ are orthonormal for $(T_x M, g)$ and the $\varepsilon_j$ are orthonormal for $(T_y F, e^{-\frac{2f(x)}{m}}h)$. The final Hermitian metric on $\bar{\Sigma}(M \times F)$ is then given by averaging as in \eqref{eq:averaging}:
\begin{equation}
    \langle \phi \otimes \nu, \phi' \otimes \nu' \rangle_{\bar{\Sigma}(M \times F)} = \frac{1}{|\Gamma_{x,y}|} \sum_{\tau \in \Gamma_{x,y}} (\tau \cdot (\phi \otimes \nu), \tau \cdot (\phi' \otimes \nu'))_{\bar{\Sigma}(M \times F)}.
\end{equation}
We leave it to the reader to check that $\langle \cdot,\cdot \rangle_{\bar{\Sigma}(M \times F)}$ actually coincides with $(\cdot,\cdot)_{\bar{\Sigma}(M \times F)}$. (This follows from the fact that $\langle \cdot,\cdot \rangle_{\pi_1^*({}^\diamond\mathbf{\Sigma}M)}$ and $\langle \cdot,\cdot \rangle_{\Sigma V}$ are obtained from the averaging procedure, and averaging such a metric does not yield a different metric.) Thus we have:
\begin{proposition} \label{prop:warped-spinor-metric}
    The Hermitian metric on $\bar{\Sigma}(M \times F) = \pi_1^*({}^\diamond \mathbf{\Sigma}M) \otimes \Sigma V$ is given by
    \begin{equation}
        \langle \phi \otimes \nu, \phi' \otimes \nu' \rangle_{\bar{\Sigma}(M \times F)} = \langle \phi, \phi' \rangle_{\pi_1^*({}^\diamond\mathbf{\Sigma}M)} \langle \nu, \nu' \rangle_{\Sigma V}
    \end{equation}
    on decomposable spinors, and extending linearly.
\end{proposition}

\subsection{The warped product spin connection and Dirac operator} \label{subsec:warped-Dirac}

In this subsection, we compute the connection and Dirac operator on $\bar{\Sigma}(M \times F)$ by adapting the discussion in \cite{roos2020dirac}. Recalling the orthogonal decomposition $(T(M \times F),\bar{g}) \cong (\pi_1^*TM, \pi_1^*g) \oplus V$, define the projections
\begin{equation}
    (\cdot)^H: T(M \times F) \to \pi_1^*TM, \quad (\cdot)^V: T(M \times F) \to V.
\end{equation}
Let $\bar{\nabla}$ be the Levi-Civita connection on $(T(M\times F),\bar{g})$. Define the 2-tensors $T, A$ which act on $X,Y \in T(M \times F)$ by
\begin{align}
    T(X,Y) &= (\bar{\nabla}_{X^V} Y^V)^H + (\bar{\nabla}_{X^V} Y^H)^V, \\
    A(X,Y) &= (\bar{\nabla}_{X^H} Y^V)^H + (\bar{\nabla}_{X^H} Y^H)^V.
\end{align}
These were originally introduced by O'Neill \cite{o1966fundamental} to study the curvatures of a Riemannian submersion.

We now choose convenient local frames for $(T(M \times F),\bar{g})$. Let $(\xi_1,\ldots,\xi_n)$ be a local orthonormal frame for $(TM,g)$. Pulling this back by $\pi_1$ gives $n$  local orthonormal vector fields on $T(M \times F)$ which we also call $\xi_1,\ldots,\xi_n$. Let $(\eta_1,\ldots,\eta_m)$ be a local orthonormal frame for $(TF,h)$. Pulling this back by $\pi_2$ gives $m$ local vector fields on $T(M \times F)$ which we also call $\eta_1,\ldots,\eta_m$. Define $\zeta_i = e^{\frac{f}{m}}\eta_i$ for $i=1,\ldots,m$. Then $$(\xi_1,\ldots,\xi_n,\zeta_1,\ldots,\zeta_m)$$ is a local orthonormal frame for $(T(M \times F),\bar{g})$.
We will also assume that $(\xi_1,\ldots,\xi_n)$ coincides with normal coordinates for $(M,g)$ at a point $x_0$, and that $(\eta_1,\ldots,\eta_m)$ coincides with normal coordinates for $(F,h)$ at $y_0$. Using that $\bar{g} = g \oplus e^{-\frac{2f}{m}}h$, routine computations yield the following identities at $(x_0,y_0)$:
\begin{gather}
    \bar{\nabla}_{\xi_\alpha}\xi_\beta = \bar{\nabla}_{\xi_\alpha}\zeta_i = 0, \quad \bar{\nabla}_{\zeta_i}\zeta_j = \frac{1}{m}\delta_{ij}\nabla f, \quad \bar{g}(\bar{\nabla}_{\zeta_i}\zeta_j,\zeta_k) = 0, \label{eq:connection-identities} \\
    A(\xi_\alpha,\xi_\beta) = A(\xi_\alpha,\zeta_i)=0, \quad T(\zeta_i,\zeta_j) = \frac{1}{m}\delta_{ij}\nabla f. \label{eq:connection-identities1}
\end{gather}

\begin{definition} \label{def:split-on-frame}
    We call the frame $(\xi_1,\ldots,\xi_n,\zeta_1,\ldots,\zeta_m)$ constructed above a \emph{split orthonormal frame} for $(T(M \times F),\bar{g})$ centered at the point $(x_0,y_0) \in M \times F$. Thus the identities \eqref{eq:connection-identities}, \eqref{eq:connection-identities1} hold at $(x_0,y_0)$.
\end{definition}

The next theorem computes the connection and Dirac operator on $\bar{\Sigma}(M \times F)$ by describing how they behave on spinors which are decomposable with respect to the identification \eqref{eq:spinor-id-1}. This actually computes the full connection and Dirac operator, because $\bar{\Sigma}(M \times F)$ is locally trivialized by decomposable spinors, and the connection and Dirac operator obey Leibniz-type rules. For the Dirac operator, this is
\begin{equation}
    \bar{D}(u\psi) = u\bar{D}\psi + \bar{\nabla} u \cdot \psi, \quad u \in \C^\infty(M \times F), \quad \psi \in \Gamma(\bar{\Sigma}(M \times F)).
\end{equation}
\begin{theorem} \label{thm:warped-dirac}
    Let $\psi \in \Gamma(\bar{\Sigma}(M \times F))$ be a spinor of the form $\psi = \pi_1^*\phi \otimes \pi_2^*\nu$, where $\phi \in \Gamma({}^\diamond\mathbf{\Sigma}M)$ and $\nu \in \Gamma(\Sigma F)$. For simplicity, identify $\phi$ with $\pi_1^*\phi$ and $\nu$ with $\pi_2^*\nu$. Write $\bar{\nabla}$ and $\bar{D}$ for the connection and Dirac operator, respectively, on $\bar{\Sigma}(M \times F)$. For all $X \in TM$, $Y \in TF$ we have
    \begin{align}
        \bar{\nabla}_X \psi &= \nabla_X \phi \otimes \nu, \label{eq:conn-H0} \\
        \bar{\nabla}_Y \psi &= \phi \otimes \nabla^F_Y \nu + \frac{1}{2m} Y \cdot \nabla f \cdot \psi, \label{eq:conn-V0} \\
        \bar{D}\psi &= \begin{cases}
            D_f \phi \otimes \nu + e^{\frac{f}{m}} \bar{\phi} \otimes D^F\nu & n \text{ even} \\
            D_f \phi \otimes \bar{\nu} + e^{\frac{f}{m}} \phi \otimes D^F\nu & n \text{ odd, }m \text{ even} \\
            (D_f \phi_1 \oplus -D_f \phi_2) \otimes \nu + e^{\frac{f}{m}} (\phi_2 \oplus \phi_1) \otimes D^F\nu & n,m \text{ odd},
        \end{cases} \label{eq:Dirac0}
    \end{align}
    where $D_f = D - \frac{1}{2}\nabla f \cdot$ is the weighted Dirac operator on $\Sigma M$, and $D^F$ is the Dirac operator on $\Sigma F$.
\end{theorem}
\begin{proof}
    Let $(\xi_1,\ldots,\xi_n,\zeta_1,\ldots,\zeta_m)$ be a split orthonormal frame for $(T(M \times F),\bar{g})$ centered at $(x_0,y_0) \in M \times F$, in the sense of Definition \ref{def:split-on-frame}. Since $\pi_1: (M \times F, \bar{g}) \to (M,g)$ is a Riemannian submersion, we can apply \cite[Lemma 6]{roos2020dirac} to get
    \begin{align}
        \bar{\nabla}_{\xi_\alpha} \psi &= \nabla_{\xi_\alpha}^{\mathcal{T}}\psi + \frac{1}{2} \sum_{\beta=1}^n \xi_\beta \cdot A(\xi_\alpha,\xi_\beta) \cdot \psi, \label{eq:conn-H} \\
        \bar{\nabla}_{\zeta_i} \psi &= \nabla_{\zeta_i}^Z \psi + \frac{1}{2} \sum_{j=1}^m \zeta_j \cdot T(\zeta_i,\zeta_j) \cdot \psi + \frac{1}{4} \sum_{\alpha=1}^n \xi_\alpha \cdot A(\xi_\alpha,\zeta_i) \cdot \psi, \label{eq:conn-V}
    \end{align}
    where
    \begin{alignat}{3}
        \nabla_{\xi_\alpha}^{\mathcal{T}} (\phi \otimes \nu) &:= \nabla_{\xi_\alpha} \phi \otimes \nu + \phi \otimes \nabla_{\xi_\alpha}^{\mathcal{V}}\nu, \qquad
        &&\nabla^{\mathcal{V}}_{\xi_\alpha} \nu := d\nu(\xi_\alpha) + \frac{1}{4} \sum_{j,k=1}^m \bar{g}(\bar{\nabla}_{\xi_\alpha} \zeta_j, \zeta_k) \zeta_j \hatdot \zeta_k \hatdot \nu, \\
        \nabla_{\zeta_i}^Z (\phi \otimes \nu) &:= \phi \otimes \nabla_{\zeta_i}^Z \nu,
        &&\nabla_{\zeta_i}^Z \nu := d\nu(\zeta_i) + \frac{1}{4} \sum_{j,k=1}^m \bar{g}(\bar{\nabla}_{\zeta_i}\zeta_j,\zeta_k) \zeta_j \hatdot \zeta_k \hatdot \nu. \label{eq:nabla-Z}
    \end{alignat}
    All computations in the rest of this proof are done at the point $(x_0,y_0)$ where the identities \eqref{eq:connection-identities}, \eqref{eq:connection-identities1} hold. Since $\bar{\nabla}_{\xi_\alpha}\zeta_j = 0$ and $\nu$ is constant in the $M$ directions, we have $\nabla_{\xi_\alpha}^{\mathcal{V}} \nu = 0$. Combined with the fact that $A(\xi_\alpha,\xi_\beta) = 0$, \eqref{eq:conn-H} becomes
    \begin{align}
        \bar{\nabla}_{\xi_\alpha}\psi &= \nabla_{\xi_\alpha}\phi \otimes \nu,
    \end{align}
    which implies \eqref{eq:conn-H0}. Since $T(\zeta_i,\zeta_j) = \frac{1}{m}\delta_{ij}\nabla f$ and $A(\xi_\alpha,\zeta_i) = 0$, \eqref{eq:conn-V} becomes
    \begin{align} \label{eq:conn-V1}
        \bar{\nabla}_{\zeta_i}\psi &= \phi \otimes \nabla^Z_{\zeta_i}\nu + \frac{1}{2m} \zeta_i \cdot \nabla f \cdot \psi.
    \end{align}
    We have $\bar{g}(\bar{\nabla}_{\zeta_i}\zeta_j,\zeta_k) = 0$; also, since the local orthonormal frame $(\eta_1,\ldots,\eta_n)$ for $(F,h)$ coincides with normal coordinates at $y_0$, the spin connection on $\Sigma F$ is $\nabla_{\eta_i}^F \nu = d\nu(\eta_i)$ (see \eqref{eq:spin-connection}).
    Now $\zeta_i = e^{\frac{f}{m}}\eta_i$, and $\eta_i, \nu$ are identified with their images under $\pi_2^*$, so
    \begin{align} \label{eq:nabla-Z-2}
        \nabla^Z_{\zeta_i} \nu = d\nu(\zeta_i) &= e^{\frac{f}{m}}d\nu(\eta_i) = e^{\frac{f}{m}} \nabla^F_{\eta_i} \nu = \nabla^F_{\zeta_i} \nu.
    \end{align}
    Combining this with \eqref{eq:conn-V1} gives \eqref{eq:conn-V0}. If $n$ is even, then by \eqref{eq:conn-H0}, \eqref{eq:conn-V0} and  \eqref{eq:cliff-mult-MxF} we have
    \begin{align}
        \bar{D}\psi &= \sum_{\alpha=1}^n \xi_\alpha \cdot \bar{\nabla}_{\xi_\alpha} \psi + \sum_{i=1}^m \zeta_i \cdot \bar{\nabla}_{\zeta_i} \psi \\
        &= \sum_{\alpha=1}^n (\xi_\alpha \cdot \nabla_{\xi_\alpha}\phi) \otimes \nu + \sum_{i=1}^m \zeta_i \cdot \left( \phi \otimes \nabla_{\zeta_i} \nu + \frac{1}{2} \zeta_i \cdot \nabla f \cdot \psi \right) \\
        &= D\phi \otimes \nu + \sum_{i=1}^m \bar{\phi} \otimes (\zeta_i \hatdot \nabla_{\zeta_i} \nu) - \frac{1}{2} \nabla f \cdot \psi \\
        &= D\phi \otimes \nu - \frac{1}{2}(\nabla f \cdot \phi) \otimes \nu + e^{\frac{f}{m}} \bar{\phi} \otimes \sum_{i=1}^m \eta_i \cdot \nabla^F_{\eta_i}\nu \\
        &= D_f \phi \otimes \nu + e^{\frac{f}{m}}\bar{\phi} \otimes D^F \nu,
    \end{align}
    where the third equality also uses that $\bar{\Sigma}(M \times F)$ is a module over $\Cl(T(M \times F),\bar{g})$ where $\zeta_i$ has unit length. This proves \eqref{eq:Dirac0} in the case that $n$ is even. The other cases are proved similarly, by modifying the above computation according to the Clifford multiplication rules \eqref{eq:cliff-mult-MxF}.
\end{proof}

\section{Applications of spinors on the warped product, and more} \label{sec:applications-Dirac}

Throughout this section, $(M^n,g)$ and $(F^m,h)$ are assumed to be spin manifolds, with $F$ being closed, scalar-flat, of unit volume, and admitting a nonzero parallel spinor $\nu$. For instance, we can take $F$ to be a flat torus with the appropriate spin structure. Using a weight $f \in \C^\infty(M)$, form the warped product $(M \times F, \bar{g} = g \oplus e^{-\frac{2f}{m}}h)$. Barred quantities denote those on the warped product. Spin geometry on the warped product was studied in \S\ref{sec:spin}, and we reuse the notations from there.

\subsection{A spin proof of our positive mass theorem} \label{subsec:spin-proof-PMT}

Let $\mathcal{M} = (M^n,g,e^{-f}d\vol_g,m)$, $m \in \N$ be an AE SMMS. We present an alternative proof of the following special case of Theorem \ref{thm:main-SMMS-PMT} using spinors on the warped product.
\begin{corollary} \label{cor:SMMS-PMT-1}
    Let $\mathcal{M} = (M^n,g,e^{-f}d\vol_g,m)$ be an AE SMMS of order $\tau > \frac{n-2}{2}$ with $m \in \N$, such that $M$ is spin. Also assume $f \in \C^{2,\alpha}_{-\tau}(M)$ and $R_f \in L^1(M,g)$. If $R^m_f \geq 0$, then $\mathfrak{m}(\mathcal{M}) \geq 0$, with equality if and only if $(M^n,g)$ is isometric to $(\R^n,\delta_{ij})$ and $f \equiv 0$.
\end{corollary}

Since $(M^n,g)$ is AE with $n > 2$, $M_\infty$ is simply connected so the spin structure on $(M^n,g)$ restricts to the trivial one on $M_\infty$. This induces a trivialization of the spinor bundle $\Sigma M$ over $M_\infty$, and hence a (partial) trivialization of $\bar{\Sigma}(M \times F)$ over $M_\infty \times F$. A \emph{constant spinor} means a spinor which is constant in these trivializations. The H\"older spaces $\C^{k,\alpha}_{-\tau}(\bar{\Sigma}(M \times F))$ are also defined in the obvious way, referring to the asymptotic behavior on $M$.

We first construct Witten-type spinors in the spinor bundle $\bar{\Sigma}(M \times F)$ of $(M \times F,\bar{g})$.

\begin{lemma} \label{lem:witten-type-spinor}
    Let $(M^n,g,e^{-f}d\vol_g,m)$, $m \in \N$ be an SMMS satisfying the hypotheses of Corollary \ref{cor:SMMS-PMT-1}.
    Then there exists a spinor $\psi \in \Gamma(\bar{\Sigma}(M \times F))$ such that
    \begin{itemize}
        \item $\psi = \pi_1^*\phi \otimes \pi_2^*\nu$ for some $\phi \in \Gamma({}^\diamond\mathbf{\Sigma}M)$ and $\nu \in \Gamma(\Sigma F)$,
        \item $D_f \phi = 0$, and $\nu$ is a unit norm parallel spinor,
        \item $\bar{D}\psi = 0$, and
        \item There exists $\psi_0 \in \Gamma(\bar{\Sigma}(M \times F))$ such that $\psi_0$ is constant with unit norm on $M_\infty \times F$, and $\psi-\psi_0 \in \C^{2,\alpha}_{-\tau}(\bar{\Sigma}(M \times F))$.
    \end{itemize}
\end{lemma}
\begin{proof}
    The hypotheses allow Theorem 2.5 in \cite{baldauf2022spinors} to be applied, giving a spinor $\phi \in \Gamma(\Sigma M)$ such that
    \begin{itemize}
        \item $D_f\phi = 0$, where $D_f$ is the weighted Dirac operator on $M$, and
        \item There exists $\phi_0 \in \Gamma(\Sigma M)$ such that $\phi_0$ is constant with unit norm on $M_\infty$, and $\phi - \phi_0 \in \C^{2,\alpha}_{-\tau}(\Sigma M)$.
    \end{itemize}
    Take a parallel spinor $\nu \in \Gamma(\Sigma F)$ with $|\nu|_{\Sigma F} = 1$. Then $|\pi_2^*\nu|_{\Sigma V} = 1$ by Lemma \ref{lem:SigmaV}. For simplicity, assume $n$ and $m$ are not both odd. Define
    \begin{equation} \label{eq:psi-defns}
        \psi = \pi_1^*\phi \otimes \pi_2^*\nu, \quad \psi_0 = \pi_1^*\phi_0 \otimes \pi_2^*\nu,
    \end{equation}
    which are sections of $\pi_1^*(\Sigma M) \otimes \Sigma V = \bar{\Sigma}(M \times F)$ (see Proposition \ref{prop:spinor-id}).
    Then $\psi_0$ is constant on $M_\infty \times F$, and Theorem \ref{thm:warped-dirac} gives $\bar{D} \psi = 0$. Moreover, by Proposition \ref{prop:warped-spinor-metric} we have $|\psi_0|_{\bar{\Sigma}(M \times F)} = |\phi_0|_{\Sigma M} |\nu|_{\Sigma V} = 1$ on $M_\infty \times F$. Similarly $|\psi-\psi_0|_{\bar{\Sigma}(M \times F)} = |\phi-\phi_0|_{\Sigma M}$.
    Since $\phi-\phi_0 \in \C^{2,\alpha}_{-\tau}(\Sigma M)$, we have $\psi - \psi_0 \in \C^{2,\alpha}_{-\tau}(\bar{\Sigma}(M \times F))$, as desired.

    In the remaining case where $n$ and $m$ are both odd, the argument is the same, except in \eqref{eq:psi-defns} we replace $\phi$ and $\phi_0$ by $\phi \oplus 0$ and $\phi_0 \oplus 0$ respectively to accommodate the fact that ${}^\diamond\mathbf{\Sigma}M = \Sigma M \oplus \Sigma M$.
\end{proof}

\begin{proof}[Proof of Corollary \ref{cor:SMMS-PMT-1}]
    Let $\mathcal{M} = (M^n,g,e^{-f}d\vol_g,m)$, $m \in \N$ be an AE SMMS satisfying the hypotheses of the corollary. Form the warped product $(M \times F,\bar{g})$, which has scalar curvature $R^m_f$. Let $\psi = \pi_1^*\phi \otimes \pi_2^*\nu \in \Gamma(\bar{\Sigma}(M \times F))$ be the Witten-type spinor provided by Lemma \ref{lem:witten-type-spinor}.
    The Lichnerowicz formula \eqref{eq:Lich} and symmetry of $\bar{D}$ give
    \begin{equation}
        0 = |\bar{D}\psi|^2 = \langle \bar{\nabla}^*\bar{\nabla}\psi, \psi \rangle_{\bar{\Sigma}(M \times F)} + \frac{1}{4}R^m_f |\psi|^2.
    \end{equation}
    Write $\xi = \psi_0 - \psi \in \C^{2,\alpha}_{-\tau}(\bar{\Sigma}(M \times F))$. Integrating the above by parts over $B^M_\rho \times F$, we get
    \begin{align}
        &\int_{B^M_\rho \times F} (|\bar{\nabla}\psi|^2 + \frac{1}{4}R^m_f|\psi|^2) d\vol_{\bar{g}} = \mathrm{Re}\int_{S^M_\rho \times F} \langle \psi, \bar{\nabla}_{\mathbf{n}}\psi\rangle dA_{\bar{g}} = \mathrm{Re} \sum_{i=1}^n \int_{S^M_\rho \times F} \langle \psi, \bar{\nabla}_{e_i}\psi\rangle e_i \lrcorner d\vol_{\bar{g}} \\
        &\qquad = \mathrm{Re} \sum_{i=1}^n \int_{S^M_\rho \times F} \left(\langle \psi_0, \bar{\nabla}_{e_i}\psi_0 \rangle - \langle \psi_0, \bar{\nabla}_{e_i}\xi\rangle - \langle \xi, \bar{\nabla}_{e_i}\psi_0 \rangle + \langle \xi, \bar{\nabla}_{e_i}\xi\rangle \right) e_i \lrcorner d\vol_{\bar{g}}, \label{eq:asdf}
    \end{align}
    where $(e_i)_{i=1}^n$ is a local orthonormal frame for $TM$. Since $M$ is AE and $F$ is closed, essentially the same arguments as in \cite[Appendix A]{lee1987yamabe} show that the first, third and fourth terms on the right vanish as $\rho \to \infty$. On the other hand, completing $(e_i)_{i=1}^n$ to a $\bar{g}$-orthonormal frame for $M \times F$, the second term is
    \begin{align}
        \mathrm{Re} \sum_{i=1}^n \int_{S^M_\rho \times F} \langle \psi_0, \bar{\nabla}_{e_i}\xi\rangle e_i \lrcorner d\vol_{\bar{g}} &= -\frac{1}{4} \sum_{i=1}^n \sum_{a=1}^{n+m} \int_{S^M_\rho \times F} (\del_a \bar{g}_{ai} - \del_i \bar{g}_{aa} + \O(\rho^{-2\tau-1}))|\psi_0|^2 e_i \lrcorner d\vol_{\bar{g}}.
    \end{align}
    Since $\bar{g}_{ai} \equiv 0$ whenever $i \leq n$ and $a > n$, this becomes
    \begin{align}
        \mathrm{Re}\sum_{i=1}^n \int_{S^M_\rho \times F} \langle \psi_0, \bar{\nabla}_{e_i}\xi\rangle e_i \lrcorner d\vol_{\bar{g}} &= -\frac{1}{4} \int_{S^M_\rho \times F} (\del_j \bar{g}_{ji} - \del_i \bar{g}_{aa} + \O(\rho^{-2\tau-1}))|\psi_0|^2 e_i \lrcorner d\vol_{\bar{g}}
    \end{align}
    where the indices $i,j$ run over the frame for $M$ and the index $a$ runs over the full frame for $M \times F$.
    Substituting this into \eqref{eq:asdf}, taking $\rho \to \infty$, and using that $|\psi_0|= 1$ outside a compact set, we get
    \begin{align} \label{eq:10581}
        \int_{M \times F} (|\bar{\nabla}\psi|^2 + \frac{1}{4}R^m_f|\psi|^2) d\vol_{\bar{g}} = \lim_{\rho\to\infty} \frac{1}{4} \int_{S^M_\rho \times F} (\del_i \bar{g}_{ij} - \del_j \bar{g}_{aa}) e_j \lrcorner d\vol_{\bar{g}} = \frac{1}{4}\mathfrak{m}(\mathcal{M}),
    \end{align}
    where the last equality is by Definition \ref{def:SMMS-mass}. From this we see that $\mathfrak{m}(\mathcal{M}) \geq 0$ if $R^m_f \geq 0$.

    Working in a split orthonormal frame $(\xi_1,\ldots,\xi_n,\zeta_1,\ldots,\zeta_m)$ centered at $(x_0,y_0) \in M \times F$, Theorem \ref{thm:warped-dirac} gives that at $(x_0,y_0)$,
    \begin{align}
        |\bar{\nabla}\psi|^2 &= \sum_{\alpha=1}^n |\bar{\nabla}_{\xi_\alpha}\psi|^2 + \sum_{i=1}^m |\bar{\nabla}_{\zeta_i}\psi|^2 \\
        &= \sum_{\alpha=1}^n |\nabla_{\xi_\alpha}\phi \otimes \nu|^2 + \sum_{i=1}^m \left| \phi \otimes \nabla_{\zeta_i}^Z \nu + \frac{1}{2m} \zeta_i \cdot \nabla f \cdot \psi \right|^2.
    \end{align}
    Since $\nabla_{\zeta_i}^Z \nu = e^{\frac{f}{m}} \nabla_{\zeta_i}^F \nu$ (see \eqref{eq:nabla-Z-2}), $\nu$ is a unit norm parallel spinor, and $\zeta_i$ has unit norm with respect to $\bar{g}$, it follows that
    \begin{align}
        |\bar{\nabla}\psi|^2 &= \sum_{\alpha=1}^n |\nabla_{\xi_\alpha}\phi \otimes \nu|^2 + \sum_{i=1}^m \left| \frac{1}{2m} \zeta_i \cdot \nabla f \cdot \psi \right|^2 \\
        &= \sum_{\alpha=1}^n |\nabla_{\xi_\alpha}\phi \otimes \nu|^2 + \frac{1}{4m} |\nabla f|^2 |\psi|^2 \\
        &= |\nabla\phi|^2 + \frac{1}{4m}|\nabla f|^2|\phi|^2. \label{eq:10958}
    \end{align}
    The last equality is by Proposition \ref{prop:warped-spinor-metric}, which also implies $|\psi| = |\phi||\nu| = |\phi|$.
    Now suppose $R^m_f \geq 0$ and $\mathfrak{m}(\mathcal{M}) = 0$. Then \eqref{eq:10581} implies $|\bar{\nabla}\psi|^2 = 0$, which in turn implies $\nabla f = 0$ and $\nabla\phi = 0$ by \eqref{eq:10958}. Thus $f$ is constant, but since $f \in \C^{2,\alpha}_{-\tau}(M)$, it is identically zero. Applying the Ricci identity \eqref{eq:weighted-Ricci} below to the $D_f$-harmonic spinor $\phi$ implies $(M,g)$ is Ricci flat. By the Bishop--Gromov comparison theorem, any Ricci flat AE manifold is exactly Euclidean, so $(M^n,g) \cong (\R^n,\delta_{ij})$.
\end{proof}

\subsection{Weighted spin geometry identities} \label{subsec:weighted-spin-ids}

In the last subsection, we used Theorem \ref{thm:warped-dirac} and an analog of Witten's proof of the unweighted mass formula \eqref{eq:WittenFormula} applied to the warped product $(M \times F, \bar{g})$, to get the formula \eqref{eq:10581} for the mass of $\mathcal{M}$. We can use this to give a new proof of the weighted Witten formula \eqref{eq:WeightedWittenFormula} of Baldauf and Ozuch. To see this, use \eqref{eq:10581}, \eqref{eq:10958} and the fact that $|\psi| = |\phi|$ to get
\begin{align}
    \frac{1}{4}\mathfrak{m}(\mathcal{M}) &= \int_{M \times F} \left(|\nabla \phi|^2 + \frac{1}{4m}|\nabla f|^2|\phi|^2 + \frac{1}{4} R^m_f|\phi|^2 \right) d\vol_{\bar{g}} \\
    &= \int_{M \times F} \left(|\nabla\phi|^2 + \frac{1}{4}R_f|\phi|^2\right) d\vol_{\bar{g}}. \label{eq:WittenFormulaSMMS}
\end{align}
Since $(F^m,h)$ has unit volume, $d\vol_{\bar{g}} = e^{-f} d\vol_g \, d\vol_h$, and $\mathfrak{m}(\mathcal{M}) = \mathfrak{m}_f(g)$, it follows that
\begin{equation}
    \mathfrak{m}_f(g) = \mathfrak{m}(\mathcal{M}) = 4\int_{M \times F}\left( |\nabla \phi|^2 + \frac{1}{4}R_f|\phi|^2 \right) e^{-f} d\vol_g \, d\vol_h = 4\int_M \left( |\nabla\phi|^2 + \frac{1}{4}R_f|\phi|^2 \right)e^{-f} d\vol_g.
\end{equation}
Recall that $\phi \in \Gamma(\Sigma M)$ is a weighted Witten spinor, so the above formula is indeed \eqref{eq:WeightedWittenFormula}.

The same outline allows us to reprove the \emph{weighted Lichnerowicz formula} and the \emph{weighted Ricci identity}. This is done in the next two propositions.

\begin{proposition}[Weighted Lichnerowicz \cite{baldauf2022spinors,branding2022eigenvalue}] \label{prop:Lichnerowicz Formula}
    Let $(M^n,g,f)$ be a weighted spin manifold. Then
    \begin{equation} \label{eq:weighted-Lich}
        D_f^2 \phi = -\Delta_f \phi + \frac{1}{4}R_f \phi
    \end{equation}
    for all $\phi \in \Gamma(\Sigma M)$, where $\Delta_f = -\nabla^*\nabla - \nabla_{\nabla f}$ is the weighted Laplacian acting on spinors.
\end{proposition}
\begin{proof}
    Form $(M \times F, \bar{g})$ as before, and let $\nu \in \Gamma(\Sigma F)$ be a nonzero parallel spinor. 
    Applying the Lichnerowicz formula \eqref{eq:Lich} to the spinor $\psi = \phi \otimes \nu \in \Gamma(\bar{\Sigma}(M \times F))$, and recalling that $\bar{g}$ has scalar curvature $R^m_f$, we get
    \begin{align} \label{eq:lich-warped}
        \bar{D}^2 \psi &= -\bar{\Delta}\psi + \frac{1}{4} R^m_f \psi.
    \end{align}
    Take a split orthonormal frame $(\xi_1,\ldots,\xi_n,\zeta_1,\ldots,\zeta_m)$ for $(T(M \times F),\bar{g})$ centered at $(x_0,y_0) \in M \times F$. Computing at $(x_0,y_0)$, using the identities \eqref{eq:connection-identities}, \eqref{eq:connection-identities1} which hold there, and using Theorem \ref{thm:warped-dirac}, we have
    \begin{align}
        \bar{\Delta}\psi &= \sum_{\alpha=1}^n (\bar{\nabla}_{\xi_\alpha} \bar{\nabla}_{\xi_\alpha} \psi - \bar{\nabla}_{\bar{\nabla}_{\xi_\alpha} \xi_\alpha} \psi) + \sum_{i=1}^m (\bar{\nabla}_{\zeta_i} \bar{\nabla}_{\zeta_i} \psi - \bar{\nabla}_{\bar{\nabla}_{\zeta_i} \zeta_i} \psi) \\
        &= \sum_{\alpha=1}^n \nabla_{\xi_\alpha}\nabla_{\xi_\alpha}\phi \otimes \nu + \sum_{i=1}^m \left(\frac{1}{4m^2}\zeta_i \cdot \nabla f \cdot \zeta_i \cdot \nabla f \cdot \psi - \frac{1}{m}\bar{\nabla}_{\nabla f} \psi \right) \\
        &= \Delta\phi \otimes \nu - \frac{1}{4m}|\nabla f|^2 \psi - \nabla_{\nabla f}\phi \otimes \nu \\
        &= \Delta_f \phi \otimes \nu - \frac{1}{4m}|\nabla f|^2\psi.
    \end{align}
    On the other hand, Theorem \ref{thm:warped-dirac} also gives $\bar{D}^2\psi = D_f^2\phi \otimes \nu$. Thus \eqref{eq:lich-warped} becomes
    \begin{align}
        (D_f^2\phi \otimes \nu) &= -\Delta_f \phi \otimes \nu + \frac{1}{4m}|\nabla f|^2\psi + \frac{1}{4}\left( R_f-\frac{1}{m}|\nabla f|^2 \right)\psi \\
        &= \left(-\Delta_f \phi + \frac{1}{4}R_f \phi\right) \otimes \nu.
    \end{align}
    The proposition follows from this since $\nu \neq 0$ is parallel and hence nonvanishing.
\end{proof}

\begin{proposition}[Weighted Ricci \cite{baldauf2022spinors}] \label{prop:Ricci-Identity}
    Let $(M^n,g,f)$ be a weighted spin manifold. Then
    \begin{align}
        [D_f,\nabla_X]\phi &= \frac{1}{2}\Ric_f(X) \cdot \phi \label{eq:weighted-Ricci}
    \end{align}
    for all $X \in TM$ and $\phi \in \Gamma(\Sigma M)$.
\end{proposition}
\begin{proof}
    Reuse the setup in the proof of Proposition \ref{prop:Lichnerowicz Formula}.
    The usual Ricci identity (see e.g. \cite[Corollary 2.8]{bourguignon2015spinorial}) applied on $\bar{\Sigma}(M \times F)$ yields
    \begin{align} \label{eq:ricci-id1}
        -\frac{1}{2}\Ric_{\bar{g}}(X) \cdot &= \sum_{\alpha=1}^n \xi_\alpha \cdot \bar{\mathcal{R}}_{X,\xi_\alpha} + \sum_{i=1}^m \zeta_i \cdot \bar{\mathcal{R}}_{X,\zeta_i},
    \end{align}
    where $\bar{\mathcal{R}}_{Y,Z} := \bar{\nabla}_Y\bar{\nabla}_Z - \bar{\nabla}_Z\bar{\nabla}_Y - \bar{\nabla}_{[Y,Z]}$.
    We claim that $\Ric_{\bar{g}}(X) = \Ric^m_f(X)$. Indeed, $\Ric_{M \times F}(X)$ is a horizontal vector because $\Ric_{M \times F}(X,Y) = 0$ for all vertical vectors; see e.g. \cite[Proposition 9.106]{besse2007einstein}.Using that $g = \bar{g}$ and $\Ric_{\bar{g}} = \Ric^m_f$ on horizontal vectors, as well as the definitions of $\Ric_{\bar{g}}, \Ric^m_f$, we see that
    \begin{equation}
        g(\Ric_{\bar{g}}(X),X') = \bar{g}(\Ric_{\bar{g}}(X),X') = \Ric_{\bar{g}}(X,X') = \Ric^m_f(X,X') = g(\Ric^m_f(X),X')
    \end{equation}
    for all horizontal vectors $X' \in TM$. Thus $\Ric_{\bar{g}}(X) = \Ric^m_f(X)$.
    
    Since $(\xi_1,\ldots,\xi_n)$ coincide with normal coordinate vector fields for $M$ at $x_0$, we compute at that point
    \begin{align}
        \sum_{\alpha=1}^n \xi_\alpha \cdot \bar{\mathcal{R}}_{X,\xi_\alpha} + \sum_{i=1}^m \zeta_i \cdot \bar{\mathcal{R}}_{X,\zeta_i} &= \sum_{\alpha=1}^n \xi_\alpha \cdot (\bar{\nabla}_X \bar{\nabla}_{\xi_\alpha} - \bar{\nabla}_{\xi_\alpha} \bar{\nabla}_X) + \sum_{i=1}^m \zeta_i \cdot (\bar{\nabla}_X \bar{\nabla}_{\zeta_i} - \bar{\nabla}_{\zeta_i} \bar{\nabla}_X - \bar{\nabla}_{[X,\zeta_i]}) \\
        &= \sum_{\alpha=1}^n (\bar{\nabla}_X(\xi_\alpha \cdot \bar{\nabla}_{\xi_\alpha}) - \xi_\alpha \cdot \bar{\nabla}_{\xi_\alpha} \bar{\nabla}_X - \bar{\nabla}_X \xi_\alpha \cdot \bar{\nabla}_{\xi_\alpha}) \\
        &\quad + \sum_{i=1}^m (\bar{\nabla}_X(\zeta_i \cdot \bar{\nabla}_{\zeta_i}) - \zeta_i \cdot \bar{\nabla}_{\zeta_i} \bar{\nabla}_X - \bar{\nabla}_X\zeta_i \cdot \bar{\nabla}_{\zeta_i} - \zeta_i \cdot \bar{\nabla}_{[X,\zeta_i]}).
    \end{align}
    By \eqref{eq:connection-identities}, we have $\bar{\nabla}_X \xi_\alpha = \bar{\nabla}_X \zeta_i = 0$ and $[X,\zeta_i] = \bar{\nabla}_X\zeta_i - \bar{\nabla}_{\zeta_i}X = \frac{1}{m}(Xf)\zeta_i$. Thus
    \begin{align}
        \sum_{\alpha=1}^n \xi_\alpha \cdot \bar{\mathcal{R}}_{X,\xi_\alpha} + \sum_{i=1}^m \zeta_i \cdot \bar{\mathcal{R}}_{X,\zeta_i} &= \bar{\nabla}_X\bar{D} - \bar{D}\bar{\nabla}_X - \frac{1}{m} (Xf) \sum_{i=1}^m \zeta_i \cdot \bar{\nabla}_{\zeta_i}.
    \end{align}
    Applying this to $\psi = \phi \otimes \nu$ on both sides, then using \eqref{eq:ricci-id1} on the left and Theorem \ref{thm:warped-dirac} on the right, we get
    \begin{align}
        -\frac{1}{2}\Ric_{\bar{g}}(X) \cdot \psi &= -[D_f,\nabla_X]\phi \otimes \nu - \frac{1}{m}(Xf) \sum_{i=1}^m \zeta_i \cdot \left( \frac{1}{2m} \zeta_i \cdot \nabla f \cdot \psi \right) \\
        &= -[D_f,\nabla_X]\phi \otimes \nu + \frac{1}{2m}(Xf)\nabla f \cdot \psi.
    \end{align}
    Since $\Ric_{\bar{g}}(X) = \Ric^m_f(X)$, this gives
    \begin{align}
        [D_f,\nabla_X]\phi \otimes \nu &= \left[\frac{1}{2}\left(\Ric^m_f(X) - \frac{1}{m}(df \otimes df)(X) \right) \cdot \phi \right] \otimes \nu = \left(\frac{1}{2}\Ric_f(X) \cdot \phi\right) \otimes \nu
    \end{align}
    and the proposition follows.
\end{proof}

\subsection{Dirac spectra of closed manifolds}

Now assume additionally that $(M^n,g)$ is compact without boundary. We will see what our earlier results imply for the spectrum of the Dirac operator on $M$. Relationships between the Dirac spectra of $M$ and of fibrations over $M$ were previously studied in \cite{ammann1998dirac,lott2002collapsing,roos2020dirac} (and references therein), although their settings are rather different from ours.

An SMMS $(M^n,g,e^{-f}d\vol_g,m)$ is said to be \emph{quasi-Einstein} if $\Ric^m_f = \lambda g$ for some $\lambda \in \R$ \cite{case2012smooth}. Our first result here characterizes this property in terms of the warped product Dirac operator $\bar{D}$. Denote by $\lambda_1(\bar{D})$ the least eigenvalue of $\bar{D}$ in absolute value.
\begin{proposition} \label{prop:quasi-Einstein}
    We have $\lambda_1(\bar{D})^2 \geq \frac{n+m}{4(n+m-1)}\min R^m_f$ for all $m \in \N$. Equality holds for $m \in \N$ if and only if the SMMS $(M^n,g,e^{-f}d\vol_g,m)$ is quasi-Einstein.
\end{proposition}
\begin{proof}
    Since $\bar{g}$ has scalar curvature $R^m_f$,
    Friedrich's inequality \cite{friedrich1980eigen} (also \cite[Theorem 5.3]{bourguignon2015spinorial}) applied to the warped product says that $\lambda_1(\bar{D})^2 \geq \frac{n+m}{4(n+m-1)}\min R^m_f$, with equality if and only if $(M \times F, \bar{g})$ is Einstein. The latter holds if and only if $(M^n,g,e^{-f}d\vol_g,m)$ is quasi-Einstein, since for horizontal vectors $X,X' \in TM$ we have
    \begin{equation}
        \Ric_{\bar{g}}(X,X') = \Ric^m_f(X,X').
    \end{equation}
\end{proof}

If $D\phi = \lambda\phi$ for some $\phi \in \Gamma(\Sigma M)$, then $\bar{D}(\phi \otimes \nu) = \lambda\phi \otimes \nu$ by Theorem \ref{thm:warped-dirac}, where $\nu \in \Gamma(\Sigma F)$ is a nontrivial parallel spinor. Thus the spectrum of $D$ is contained in the spectrum of $\bar{D}$, and Proposition \ref{prop:quasi-Einstein} implies $\lambda_1(D)^2 \geq \frac{n+m}{4(n+m-1)}\min R^m_f$ for all $m \in \N$. This bound is similar to, but weaker than the weighted Friedrich inequality $\lambda_1(D)^2 \geq \frac{n}{4(n-1)}\min R_f$ \cite[Theorem 1.23]{baldauf2022spinors}. Equality in the latter inequality holds if and only if $f$ is constant, and $(M^n,g)$ admits a nontrivial Killing spinor and is Einstein. The next corollary uses this to give a similar characterization of equality in our bound.

\begin{corollary}
    We have $\lambda_1(D)^2 \geq \frac{n+m}{4(n+m-1)} \min R^m_f$ for all $m \in \N$ and $f \in \C^\infty(M)$. Equality holds for some $m$ if and only if $f$ is constant and $M$ admits a nontrivial parallel spinor (hence is Ricci flat).
\end{corollary}
\begin{proof}
    Argue as in the last paragraph to get the claimed bound on $\lambda_1(D)^2$. Now suppose equality holds in the bound for some $m \in \N$. By the weighted Friedrich inequality, we have
    \begin{equation}
        \lambda_1(D)^2 \geq \frac{n}{4(n-1)}\min R_f \geq \frac{n+m}{4(n+m-1)} \min R^m_f,
    \end{equation}
    but the equality assumption turns both inequalities above into equalities. The first equality implies $f$ is constant and $M$ is Einstein, so $R^m_f = R_f = R$. This, the second equality, and the fact that $\frac{n}{4(n-1)} > \frac{n+m}{4(n+m-1)}$ imply that $\min R_f \leq 0$. The first equality in turn forces $\lambda_1(D) = \min R_f = 0$, so $M$ admits a nontrivial parallel spinor. Finally, a spin manifold admitting a nontrivial parallel spinor must be Ricci-flat (see e.g. \cite[\S 3.2]{friedrich2000dirac}).

    Conversely, if $f$ is constant and $M$ admits a nontrivial parallel spinor, then $\lambda_1(D)^2 = 0$. On the other hand $(M^n,g)$ is Ricci-flat, hence $R^m_f = R = 0$. Thus $\lambda_1(D)^2 = \frac{n+m}{4(n+m-1)}\min R^m_f$.
\end{proof}

Specializing our discussion to harmonic spinors, first note that the weighted Lichnerowicz formula
$D_f^2 = \nabla^*\nabla + \frac{1}{4}R_f$
implies the following, since $R_f \geq R^m_f$ for $m > 0$:
\begin{corollary} \label{cor:vanishing-easy}
    If there exist $m >0 $ and $f \in \C^\infty(M)$ with $R^m_f \geq 0$ and $R^m_f > 0$ at some point, then $M$ admits no nontrivial harmonic spinors.
\end{corollary}
Corollary \ref{cor:vanishing-thm} extends Corollary \ref{cor:vanishing-easy} to include $m < 1-n$; this is not implied by the weighted Lichnerowicz formula since $R^m_f \leq R_f$ here. For its proof, we make the following observation. In addition to $(M,g)$ and its spinor bundle $\Sigma M$, consider the conformally changed manifold $(M,\tilde{g} = e^{-\frac{2f}{n-1}}g)$ and its spinor bundle $\tilde{\Sigma}M$. According to \cite[\S 2.3.5]{bourguignon2015spinorial}, there is a natural bundle isometry $\Phi: \Sigma M \to \tilde{\Sigma}M$, and the Dirac operators $D,\tilde{D}$ of the respective bundles are related by
\begin{equation} \label{eq:Df-D-relationship}
    D_f = D - \frac{1}{2}\nabla f \cdot = e^{-\frac{f}{n-1}} \Phi^{-1} \tilde{D} \Phi.
\end{equation}
Thus, the weighted Dirac operator on $(M,g,f)$ is obtained from the usual Dirac operator on $(M,\tilde{g})$ by conjugating by $\Phi$ and multiplying by a function.
Interestingly, $\tilde{g}$ is the same metric used in \S\ref{sec:PMTs} to prove our positive mass theorems.

\begin{proof}[Proof of Corollary \ref{cor:vanishing-thm}]
    By Lemma \ref{lem:R-of-conformal} (or \eqref{eq:Rtilde-SMMS}), the assumption that $R_f^m \geq 0$ and is positive somewhere imply that the scalar curvature $\tilde{R}$ of $\tilde{g}$ is $\geq 0$ and is positive somewhere. The Lichnerowicz formula applied to the spinor bundle $\tilde{\Sigma}M$ of $(M,\tilde{g})$ shows that the Dirac operator $\tilde{D}$ on $\tilde{\Sigma}M$ has trivial kernel. Then \eqref{eq:Df-D-relationship} shows that $D_f$ has trivial kernel on $\Sigma M$. As observed independently in \cite{baldauf2022spinors} and \cite{branding2022eigenvalue}, $D$ and $D_f$ have the same eigenvalues. Thus $D$ also has trivial kernel.
\end{proof}

\begin{remark}
    Corollary \ref{cor:vanishing-thm} cannot be extended to any $m \in [-1,0)$. In fact, \emph{every} closed manifold $M$ of dimension $\geq 3$ admits a metric $g$ and a function $f$ for which $R^m_f > 0$ whenever $m \in [-1,0)$.
    Indeed, by \cite[Theorem 3]{eliasson1971variations} there exists a metric $g$ on $M$ with $\int_M R \, d\vol_g > 0$,
    so the function
    \begin{equation}
        u := \frac{R}{2} - \frac{\int_M R \, d\vol_g}{2\vol(M,g)}
    \end{equation}
    satisfies $\int_M u \, d\vol_g = 0$. Thus $u = -\Delta f$ for some $f \in \C^\infty(M)$, so for all $m \in [-1,0)$ we have
    \begin{equation}
        R^m_f = R + 2\Delta f - \frac{m+1}{m}|\nabla f|^2 \geq R - 2u = \frac{\int_M R \, d\vol_g}{\vol(M,g)} > 0.
    \end{equation}
    Nonetheless it remains unclear whether Corollary \ref{cor:vanishing-thm} can be extended to also include all values of $m$ in the interval $(1-n,-1)$.
\end{remark}

Corollary \ref{cor:vanishing-thm} motivates the problem of finding $f \in \C^\infty(M)$ and $m \in \R\setminus[1-n,0]$ such that $R^m_f$ is a constant $\mu_m$, since if $\mu_m > 0$ then the corollary implies $M$ has no nontrivial harmonic spinors. The next proposition shows that this problem can be solved, and $\mu_m > 0$ when $R > 0$. While no new obstructions to harmonic spinors arise from this, we will see shortly that the constants $\mu_m$ relate to well-known lower bounds for the Dirac spectrum.

\begin{proposition}
    Let $(M,g)$ be a closed Riemannian manifold. For each $m \in \R \setminus \{0\}$, there is a unique constant $\mu_m \in \R$ and a smooth function $f \in \C^\infty(M)$ unique up to translation, such that $R^m_f = \mu_m$. In fact
    \begin{equation} \label{eq:principal}
        \mu_m = \inf_{\substack{v \in H^1(M,g), \\ \lVert v \rVert_{L^2(M,g)} = 1}} \int_M \left( \frac{4m}{m+1}|\nabla v|^2 + Rv^2 \right)\, d\vol_g.
    \end{equation}
\end{proposition}
\begin{proof}
    Inspired by \cite[Theorem 2]{dobarro1987scalar}, we set $u := e^{-\frac{m+1}{2m}f}$. A careful computation yields
    \begin{equation}
        -\frac{4m}{m+1} \Delta_g u + Ru = R_f^m u.
    \end{equation}
    Thus, solving $R^m_f = \mu_m$ is equivalent to finding a positive solution $u$ to the eigenvalue problem
    \begin{equation}
        Lu := \left(-\frac{4m}{m+1} \Delta_g + R \right)u = \mu_m u.
    \end{equation}
    By standard elliptic theory, only the principal eigenfunctions of $L$ do not change sign; thus, $\mu_m$ must equal the principal eigenvalue, which is given by \eqref{eq:principal}. Also the $\mu_m$-eigenspace is one-dimensional and consists of smooth functions. Thus, $u$ exists uniquely up to positive scaling, and $f$ is unique up to translation.
\end{proof}

Now let $(M^n,g)$ be a closed spin manifold with Dirac operator $D$. Write $\lambda_1(\cdot)$ for the lowest eigenvalue of a second-order operator. Friedrich's \cite{friedrich1980eigen} and Hijazi's \cite{hijazi1986conformal} inequalities respectively imply that
\begin{align}
    \lambda_1(D)^2 &\geq \frac{n}{4(n-1)} \lambda_1(-4\Delta + R), \label{eq:Friedrich} \\
    \lambda_1(D)^2 &\geq \frac{n}{4(n-1)}\lambda_1(\Box), \label{eq:Hijazi}
\end{align}
where $\Box = -\frac{4(n-1)}{n-2}\Delta + R$ is the conformal Laplacian.
Meanwhile, observe using \eqref{eq:principal} that
\begin{align}
    \mu_m \uparrow \lambda_1(-4\Delta + R) & \quad \text{as } 0 < m \uparrow \infty, \\
    \mu_m \uparrow \lambda_1(\Box) & \quad \text{as } -\infty < m \uparrow 1-n.
\end{align}
Therefore, the inequalities \eqref{eq:Friedrich} and \eqref{eq:Hijazi} are implied by the family of weaker inequalities
\begin{equation}
    \lambda_1(D)^2 \geq \frac{n}{4(n-1)}\mu_m, \quad m \in \R \setminus [1-n,0].
\end{equation}
This family of inequalities interpolates between the inequalities of Friedrich and Hijazi.

\bibliographystyle{plain} 
\bibliography{Refs} 

\end{document}